\documentclass[letterpaper, 11pt]{amsart}
\usepackage{mathtools}
\usepackage{amsmath}
\usepackage{amssymb}
\usepackage{yhmath}
\usepackage{graphicx}
\usepackage{xcolor} % Blue additions and pink deletions.
\usepackage{mathrsfs}
\usepackage{bbm}
\usepackage{tikz-cd}
\usepackage{tikz}
\usetikzlibrary{patterns}
\usepackage{hyperref}
\hypersetup{hidelinks} % Keep the revision colors free of link boxes.

\usepackage{verbatim}
\usepackage{float}

\DeclareMathAlphabet{\mathpzc}{OT1}{pzc}{m}{it}

\usepackage{thmtools}
\usepackage{thm-restate}

\usepackage{caption}

\newtheorem{theorem}{Theorem}[section]

\newtheorem*{claim*}{Claim}

\newtheorem{lemma}[theorem]{Lemma}

\newtheorem{corollary}[theorem]{Corollary}

\newtheorem{proposition}[theorem]{Proposition}

\newtheorem{prop}[theorem]{Proposition}

\theoremstyle{definition}
\newtheorem{definition}[theorem]{Definition}
\newtheorem{Def}[theorem]{Definition}
\newtheorem{example}[theorem]{Example}

\theoremstyle{remark}
\newtheorem{remark}[theorem]{Remark}

\numberwithin{equation}{section}

\newcommand{\norm}[1]{\lVert#1\rVert}
\newcommand{\op}{\operatorname}

\newcommand{\be}{\begin{equation}}
\newcommand{\ee}{\end{equation}}
\newcommand{\Ga}{\Gamma}

\newcommand{\R}{\mathbb R}
\renewcommand{\H}{\mathbb H}
\newcommand{\Z}{\mathbb Z}
\newcommand{\N}{\mathbb N}
\newcommand{\ga}{\gamma}

\newcommand{\la}{\lambda}
\newcommand{\La}{\Lambda}

\newcommand{\cal}{\mathcal}
\newcommand{\br}{\mathbb R}
\newcommand{\SO}{\op{SO}}
\newcommand{\Isom}{\op{Isom}}
\newcommand{\PSL}{\op{PSL}}
\newcommand{\F}{\cal F}

\newcommand{\hull}{\op{hull}}
\newcommand{\stab}{\op{Stab}}

\newcommand{\diam}{\op{diam}}

\renewcommand{\frak}{\mathfrak}

\newcommand{\e}{\varepsilon}

\renewcommand{\L}{\mathcal L}
\newcommand{\fa}{\mathfrak a}

\renewcommand{\i}{\op{i}}

\newcommand{\so}{\SO^\circ}

\newcommand{\id}{\op{id}}

\newcommand{\fg}{\frak g}

\newcommand{\Lie}{\op{Lie}}

\newcommand{\lat}{\La_{\theta}}
\newcommand{\ft}{\F_\theta}

\renewcommand{\epsilon}{\e}

\renewcommand{\d}{\mathsf{d}}

\newcommand{\SL}{\op{SL}}

 \title[Global shadow estimates and Hausdorff measure criteria]{Global shadow estimates and Hausdorff~measure criteria in higher rank}

\author{Dongryul M. Kim}
\address{School of Mathematics, Institute for Advanced Study,  Princeton, NJ 08540 (Current)
\newline \indent Simons Laufer Mathematical Sciences Institute, Berkeley, CA 94720
\newline \indent Department of Mathematics, Yale University, New Haven, CT 06511}
\email{dongryul.kim97@gmail.com}

\author{Hee Oh}
\address{Department of Mathematics, Yale University, New Haven, CT 06511}
\email{hee.oh@yale.edu}

\thanks{Oh is partially supported by the NSF grant No. DMS-2450703.}

\begin{document}

\begin{abstract}

Patterson--Sullivan measures encode the distribution of orbits of
discrete group actions near the boundary. We prove a global shadow
lemma for Patterson--Sullivan measures associated to relatively
Morse subgroups of higher-rank semisimple Lie groups. The estimate
is uniform for shadows centered at arbitrary points in a Gromov
model, including points deep in the cuspidal part, and expresses
the contribution of each cusp in terms of peripheral growth.
This extends the global shadow lemma of Stratmann--Velani for
geometrically finite real hyperbolic groups.
As applications, we establish local doubling and reverse
doubling. Under Zariski density, we give necessary and sufficient criteria for
Patterson--Sullivan measures to agree  with the
associated Hausdorff measures and to be Ahlfors regular. For symmetric linear forms, these criteria hold without
Zariski density.
We verify the criteria explicitly for cusped Hitchin
representations and real-hyperbolic self-joinings.
\end{abstract}

\maketitle

\tableofcontents

\section{Introduction}

Patterson--Sullivan measures relate orbit growth to the geometry
of limit sets. In the presence of cusps, their local behavior
also reflects the growth of peripheral subgroups. A global
shadow estimate describes this dependence uniformly, including
for shadows  centered  arbitrarily deep in cusp horoballs.
In this paper we establish such an estimate for relatively
Morse subgroups of higher-rank semisimple Lie groups.
We use it to give criteria characterizing when these measures
agree, up to scale, with the associated Hausdorff measures
and when they are Ahlfors regular. 
We give explicit
applications to cusped Hitchin representations and
real-hyperbolic self-joinings.
We begin with the classical real hyperbolic case.

 Let
$\Gamma < \Isom^+(\H^n)$ be a non-elementary discrete subgroup, let
$o \in \H^n$, and let $\Lambda \subset \partial \H^n$ be the limit set of
$\Gamma$.  Patterson \cite{Patterson1976limit} and Sullivan
\cite{Sullivan1979density} constructed a Borel probability measure $\nu$ supported on
$\Lambda$ whose transformation rule is governed by the critical exponent
$\delta_\Gamma > 0$: 
$$
\frac{d\gamma_*\nu}{d\nu}(\xi)
=
e^{\delta_\Gamma \beta_\xi(o,\gamma o)}
\quad \text{for all } \gamma \in \Gamma
\text{ and } \nu\text{-a.e. } \xi \in \Lambda .
$$
Here $\beta_\xi$ denotes the Busemann function.  The measure $\nu$ is now referred to as the
Patterson--Sullivan measure of $\Gamma$.

A fundamental feature of this measure is Sullivan's shadow lemma.  For
$x,y\in \H^n$ and $R>0$, let $O_R(x,y)\subset \partial \H^n$ denote the
shadow of the ball $B(y,R)$ seen from $x$, namely the set of endpoints
$\xi$ such that the geodesic ray $[x,\xi] \subset \H^n$ intersects the $R$-neighborhood of
$y$.

\begin{theorem}[{Shadow Lemma \cite[Proposition 3]{Sullivan1979density}}] \label{thm.Sullivanshadowlemma}
For all sufficiently large $R>0$, there exists $C>1$ such that
$$
C^{-1}e^{-\delta_\Ga d(o,\ga o)}\le \nu(O_R(o,\ga o))\le C e^{-\delta_\Ga d(o,\ga o)}
\qquad (\ga\in \Ga ).
$$
\end{theorem}
Thus the Patterson--Sullivan measure of a shadow is comparable to the
exponential of minus the orbit distance:
$
\nu(O_R(o,\gamma o)) \asymp e^{-\delta_\Gamma d(o,\gamma o)},
$
with implied constants independent of $\gamma$.  This estimate is the basic
prototype for the global shadow lemma proved in this paper.

\subsection*{Geometrically finite groups}

For geometrically finite groups, Sullivan's shadow lemma admits a global
version due to Stratmann--Velani \cite{SV_globalshadow}.  Let
$\Gamma<\Isom^+(\H^n)$ be geometrically finite with limit set
$\Lambda \subset \partial \H^n$, and let $\hull\Lambda\subset \H^n$ be its convex hull.  Choose
a $\Gamma$-invariant family $\cal B$ of pairwise disjoint cusp horoballs in
$\hull\Lambda$ whose complement has compact quotient.  Let $\cal P$ be a
set of representatives of the corresponding maximal parabolic subgroups,
and write $\delta_{\mathsf P}$ for the critical exponent of
$\mathsf P\in\cal P$. It is known that $\delta_{\mathsf P}<\delta_{\Ga}$ (\cite{Sullivan1979density}, \cite{Sullivan1984entropy}).

\begin{theorem}[{Global Shadow Lemma \cite[Theorem 2]{SV_globalshadow}}]
\label{thm.SVglobal}
For all sufficiently large $R>0$, the following holds.  Let
$\xi\in\Lambda$ and $x\in[o,\xi]$.  If $x$ lies in a horoball
$B\in\cal B$ whose stabilizer is conjugate to $\mathsf P\in\cal P$, then
$$
\nu(O_R(o,x))
\asymp
e^{-\delta_\Gamma d(o,x)}
e^{(2\delta_{\mathsf P}-\delta_\Gamma)d(\Gamma o,x)},
$$
with implied constants independent of $\xi$, $x$, and $B$.
\end{theorem}

In $\hull \La$ but away from the cusp horoballs, the quantity $d(\Gamma o,x)$ is uniformly
bounded, and the estimate reduces to the usual shadow lemma
$\nu(O_R(o,x))\asymp e^{-\delta_\Gamma d(o,x)}$.

\subsection*{Relatively Anosov and relatively Morse groups}
Relatively Anosov subgroups provide a higher-rank analogue of
geometrically finite Kleinian groups.  Let $G$ be a connected semisimple
real algebraic group with Cartan decomposition $G=KA^+K$ where $K$ is a maximal compact subgroup.  Let
$X:=G/K$ denote the associated Riemannian symmetric space.  Let $\fa^+:=\log A^+$, let $\Pi$ be the corresponding set of
simple roots, and fix a non-empty subset $\theta\subset\Pi$.  We denote by
$P_\theta$ the associated standard parabolic subgroup and by
\[
    \F_\theta:=G/P_\theta
\]
the corresponding flag variety.  We also set
\[
    \fa_\theta := \bigcap_{\alpha\in\Pi-\theta}\ker\alpha
\]
and regard $\fa_\theta^*$ as a subspace of $\fa^*$ via the canonical projection
$p_\theta:\fa\to\fa_\theta$.  After replacing $\theta$ by
$\theta\cup\i(\theta)$, if necessary, we assume that
\[
    \theta=\i(\theta),
\]
where $\i$ is the opposition involution on $\fa$.

Let $\Ga<G$ be a discrete subgroup which is hyperbolic relative to a finite
collection $\cal P$ of finitely generated infinite subgroups.  We write
$\partial(\Ga,\cal P)$ for its Bowditch boundary.  A Gromov model for
$(\Ga,\cal P)$ is a proper geodesic Gromov hyperbolic space $Y$ on which
$\Ga$ acts properly discontinuously, together with a $\Ga$-invariant
collection $\cal B$ of disjoint horoballs whose stabilizers are conjugates
of subgroups in $\cal P$, and on whose complement $\Ga$ acts cocompactly.
We identify $\partial Y$ with $\partial(\Ga,\cal P)$, and we assume that it contains at least three points, i.e., $(\Ga, \cal P)$ is non-elementary.

We say that $\Ga$ is $\theta$-Anosov relative to $\cal P$ if it is
$\theta$-regular and admits a transverse $\Ga$-equivariant boundary map
\[
    \zeta:\partial Y\to\F_\theta .
\]
See section \ref{sec.relmorse} for its precise definition.
The image of this map is the $\theta$-limit set, denoted by
$\La_\theta$.

For $\psi\in\fa_\theta^*$, a $(\Ga,\psi)$-Patterson--Sullivan measure is a
Borel probability measure $\nu$ on $\La_\theta$ satisfying
\[
    \frac{d\gamma_*\nu}{d\nu}(\xi)
    =
    e^{\psi(\beta^\theta_\xi(e,\gamma))}
    \quad
    \text{for all } \gamma\in\Ga
    \text{ and } \nu\text{-a.e. } \xi\in\La_\theta ,
\]
where $\beta^\theta$ is the $\fa_\theta$-valued Busemann map; see
\eqref{Bu}.  This notion of higher-rank Patterson--Sullivan measure was introduced by Quint
\cite{Quint2002Mesures}.

Let $\mu : G \to \fa^+$ be the Cartan projection.
Let $\L_\Ga\subset\fa^+$ denote the asymptotic cone of the Cartan projection $\mu(\Ga)$ of $\Ga$, called the limit cone of $\Ga$.
The shadow lemma for Patterson--Sullivan measures was proved for shadows
in the higher-rank symmetric space $X=G/K$ in
\cite[Lemma 7.8]{LO_invariant} and \cite[Lemma 7.2]{KOW_indicators}.
For relatively $\theta$-Anosov subgroups, the compatibility between shadows
in the Gromov model $Y$ and shadows in $X$ was established in
\cite[Proposition 5.7]{KO_entropy}.\footnote{The reference
\cite{KO_entropy} treats a specific Gromov model, but the same proof
applies to the Gromov models considered here.}  Combining these results
gives the following orbit-shadow estimate.

\begin{theorem}[Shadow Lemma] \label{thm.shadowlemmahigherrank}
Let $\Ga<G$ be $\theta$-Anosov relative to $\cal P$, and fix
$o_Y\in Y$.  
For $\psi \in \fa_{\theta}^*$, let $\nu$ be a $(\Ga,\psi)$-Patterson--Sullivan measure on $\La_\theta$.
Then, for all sufficiently large $R>0$,
\[
    \nu\bigl(\zeta(O_R(o_Y,\gamma o_Y))\bigr)
    \asymp
    e^{-\psi(\mu(\gamma))} \qquad (\gamma\in\Ga),
\]
 with implied constants independent of $\gamma$.
\end{theorem}

The  relatively Morse condition strengthens the relative Anosov condition by
requiring the relative geometry of $Y$ to be realized inside the
symmetric space.  We say that $\Ga$ is $\theta$-Morse relative to
$\cal P$ if there exists a $\Ga$-equivariant quasi-isometric embedding
\[
    f:Y\to X
\]
such that, writing $\L_f\subset\fa^+$ for the asymptotic cone of the
Cartan projections
\[
    \{ \mu(f(x)^{-1}f(y)) : x,y\in Y \},
\]
we have
\[
    \L_f\cap\ker\alpha=\{0\}
    \quad\text{for every } \alpha\in\theta .
\]
Thus being relative Anosov gives a boundary map into the flag variety,
whereas the relatively Morse condition gives a coarse geometric model in
the symmetric space whose Cartan projections stay uniformly away from the
walls indexed by $\theta$.  By the higher-rank Morse lemma of
Kapovich-Leeb-Porti \cite[Theorem 1.4]{KLP_2018}, the map $f$ extends continuously to a transverse
$\Ga$-equivariant homeomorphism
\[
    f:\partial Y\to\La_\theta.
\]

Fix a basepoint $o_Y\in Y$.  For $\psi\in\fa_\theta^*$ positive on
$\L_f-\{0\}$, define
\[
    \d_\psi(x,y)
    : =
    \psi(\mu(f(x)^{-1}f(y)))
     \quad \text{for } x, y \in Y,
\]
and, for subsets $E,F\subset Y$, set
\[
    \d_\psi(E,F)
    :=
    \inf_{x\in E,\,y\in F}\d_\psi(x,y).
\]

Theorem \ref{thm.shadowlemmahigherrank} estimates shadows centered at
orbit points.  The main result of this paper is a global version  for
relatively Morse groups, where the center of the shadow may be an
arbitrary point along a geodesic ray in the Gromov model.

 For a subgroup $H<\Ga$, denote by $\delta_\psi(H)$ the critical exponent of the Poincar\'e
series
\[
    s\mapsto \sum_{\gamma\in H} e^{-s\psi(\mu(\gamma))}.
\]
In the setting  below, the strict inequality $\delta_{\psi}(\mathsf{P}) < \delta_{\psi}(\Ga)$ was proved by Canary--Zhang--Zimmer \cite{CZZ_relative}.

The following theorem is the
higher-rank relatively Morse analogue of the global shadow lemma of Stratmann--Velani \cite{SV_globalshadow}. 
\begin{theorem}[Global Shadow Lemma] \label{thm.introglobal}
Let $\Ga<G$ be $\theta$-Morse relative to $\cal P$, and let $\psi\in\fa_\theta^*$ be such that $\psi>0$ on $\L_f - \{0\}$ and $
    \delta_\psi(\Ga)=1$.
Let $\nu$ be a $(\Ga,\psi)$-Patterson--Sullivan measure on $\La_\theta$.
Then there exists $C_0>0$ such that, for all sufficiently large $R>0$, the
following holds.

Let $\xi\in\partial Y$ and $x\in[o_Y,\xi]$.  Suppose that $x\in B$ for
some horoball $B\in\cal B$ whose stabilizer is conjugate to
$\mathsf P\in\cal P$.  Then
\[
\begin{aligned}
\nu(f(O_R(o_Y,x)))
\asymp{}&
e^{-\d_\psi(o_Y,x)}
e^{\d_\psi(\Ga o_Y,x)}
e^{2(\delta_{\bar\psi}(\mathsf P)-1)
    \d_{\bar\psi}(\Ga o_Y,x)}
\\
&\cdot
\bigl(C_0+\d_{\bar\psi}(\Ga o_Y,x)\bigr)^{a_{\bar\psi}(\mathsf P)},
\end{aligned}
\]
where $\bar\psi=\frac{\psi+\psi\circ\i}{2}$
and $a_{\bar\psi}(\mathsf P)$ is a non-negative integer depending on $\bar\psi$ and $\mathsf P$.  The implied
constants are independent of $\xi$, $x$, and $B$.
\end{theorem}

See Remark~\ref{rmk.ballestimate} for estimates in terms of the
visual quasi-metric on $\La_\theta$.

Here $a_{\bar\psi}(\mathsf P)$ is the exponent of the polynomial
factor in the growth of $\mathsf P$ measured by $\bar\psi$.
It vanishes when $\mathsf P$ is virtually cyclic
(Theorem~\ref{thm.countinginsegments}).
It would be interesting to know whether the relatively Morse
hypothesis forces $a_{\bar\psi}(\mathsf P)=0$ for every
$\mathsf P\in\cal P$.

If $x$ lies in the thick part
$Y-\bigcup_{B\in\cal B}B $,
then Theorem \ref{thm.introglobal}
recovers the usual orbit-shadow estimate.  The theorem 
therefore identifies the precise additional factor that appears when $x$ enters a
cusp.  
\begin{remark}
Bray--Tiozzo \cite[Theorem 1.4]{BT_global} proved a global
shadow lemma for geometrically finite actions on proper
hyperbolic metric spaces under a tempered parabolic growth
assumption. Their measures are quasi-conformal with respect
to the Busemann functions of the hyperbolic space.
In our setting, although shadows are taken in the Gromov
model, the Busemann maps, Patterson--Sullivan measures,
and critical exponents come from the ambient higher-rank
Lie group.
\end{remark}

\begin{example}
The relatively Morse hypothesis holds for the following families.

    \begin{enumerate}
        \item Let $G = \prod_{i = 1}^k \so(n_i, 1)$, $n_i \ge 2$. It follows from the work of Bowditch \cite{Bowditch1998} and Yaman \cite{Yaman2004topological} that a discrete subgroup $\Ga < G$ is relatively $\Pi$-Anosov if and only if there exists a geometrically finite subgroup $\Ga_1 < \so(n_1, 1)$ and geometrically finite type-preserving representations $\rho_i : \Ga_1 \to \so(n_i, 1)$, $2 \le i \le k$, so that the diagonal embedding $(\id \times \rho_2 \times \cdots \times \rho_k)(\Ga_1)$ is a finite-index subgroup of~$\Ga$.
The work of Tukia \cite{Tukia1985isomorphisms} (also see \cite{DSU}) then implies that $\Ga$ is relatively $\Pi$-Morse. Moreover, there exists a Morse embedding $f$ such that $\L_f = \L_{\Ga}$.

        \item We consider a relatively $\theta$-Anosov $\Ga < \SL(n, \R)$, $n\ge 2$, with peripheral subgroups $\cal P$, such that $(\Ga, \cal P)$ is isomorphic to  a geometrically finite Fuchsian group. In this case, Zhu--Zimmer showed that $\Ga$ is relatively $\theta$-Morse (\cite[Corollary 1.14]{ZZ_relatively2}, \cite[Proposition 11.3]{ZZ_relatively}). They explicitly constructed a Morse embedding $f$ such that $\L_f = \L_{\Ga}$ in \cite[Sections 6, 10, 11]{ZZ_relatively}.
    \end{enumerate}
\end{example}
\subsection*{Local properties of Patterson--Sullivan measures}

The global shadow lemma also gives local information on Patterson--Sullivan
measures. For distinct
$\xi,\eta\in\Lambda_\theta$, set
\[
    d_\psi(\xi,\eta):=e^{-\psi(\cal G^\theta(\xi,\eta))}
\]
where $\cal G^\theta$ is the $\fa_\theta$-valued Gromov product.  This is
a higher-rank visual quasi-metric on $\Lambda_\theta$; denote by
$B_\psi(\xi,r)$ the corresponding balls.
As applications, we obtain the following uniform local estimates for
Patterson--Sullivan measures.

 \begin{corollary}[Uniform local estimates]
Let $\Ga$, $\psi$, $\nu$ be as in Theorem~\ref{thm.introglobal}.  For every $\kappa\ge 1$, there exists $L_\kappa>1$ and $C_\kappa >1$ such that
\[C_\kappa^{-1} \nu(B_\psi(\xi, L_\kappa r) )\le \nu(B_\psi(\xi, r)) \le \kappa^{-1}\nu(B_\psi (\xi, L_\kappa r) )\]
for all $\xi\in \Lambda_\theta$ and all $0<r\le L_\kappa^{-1}$.
\end{corollary}
This corollary follows immediately from Theorem \ref{thm.doubling}.

Finally, the global estimate gives a sharp criterion for the
Patterson--Sullivan measure to be a Hausdorff measure. Denote by
$\cal H_\psi$ the one-dimensional Hausdorff measure associated to
$d_\psi$. 
For Anosov groups, Dey--Kim--Oh \cite[Theorem 1.1]{DKO_AR}
proved that Patterson--Sullivan measures associated to 
symmetric linear forms are Ahlfors $1$-regular and agree, up to
scale, with the associated Hausdorff measures.
Our results recover their theorem when $\cal P=\emptyset$
and determine precisely how peripheral growth affects these
conclusions in the presence of cusps. In particular,
identification with Hausdorff measure can persist even when
Ahlfors regularity fails.

\begin{theorem}[Hausdorff measure]\label{thm.introHmeasure}
Let $\Ga,\psi,\nu$ be as in Theorem~\ref{thm.introglobal}.
Then $\cal H_\psi(\La_\theta)<\infty$.
If $\Ga$ is Zariski dense, then $\nu$ is a positive multiple
of $\cal H_\psi$ if and only if $\psi=\psi\circ\i$ and,
for every $\mathsf P\in\cal P$, either
\[
    \delta_\psi(\mathsf P)<\frac12,
    \quad\text{or}\quad
    \delta_\psi(\mathsf P)=\frac12
    \ \text{and}\ a_\psi(\mathsf P)=0.
\]
For symmetric $\psi$, this peripheral criterion is also
necessary and sufficient without assuming that $\Ga$ is
Zariski dense. 
\end{theorem}

The theorem is proved in Section~\ref{sec.Hausdorff}.
For virtually cyclic peripheral subgroups, the peripheral
condition reduces to $\delta_\psi(\mathsf P)\le1/2$
for every $\mathsf P$.
If $\Ga$ is Zariski dense, we also show that $\nu$ is Ahlfors
regular if and only if $\psi=\psi\circ\i$ and
\[
    \delta_\psi(\mathsf P)=\frac12
    \quad\text{and}\quad a_\psi(\mathsf P)=0
    \qquad\text{for every }\mathsf P\in\cal P;
\]
in that case its Ahlfors exponent is $1$
(Corollary~\ref{cor.ahlforscriterion}).
For symmetric $\psi$, this Ahlfors regularity criterion also
holds without Zariski density.
Thus, in the symmetric case, peripheral growth determines
both whether the Hausdorff measure is positive and whether
its local mass is uniformly comparable to the radius.

The hypotheses can be checked explicitly in familiar families.
For a cusped Hitchin representation of a torsion-free nonuniform
lattice in $\PSL(2,\R)$ into $\PSL(2m,\R)$, with $m\ge2$,
the middle simple root $\psi=\alpha_m$ is symmetric and satisfies
\[
    \delta_\psi(\Ga)=1,\qquad
    \delta_\psi(\mathsf P)=\frac12,\qquad
    a_\psi(\mathsf P)=0
\]
for every peripheral subgroup $\mathsf P$.
Consequently, $\nu_\psi$ is Ahlfors $1$-regular
(Corollary~\ref{cor.cuspedHitchinAhlfors}).
For self-joinings in products of real hyperbolic groups, the
peripheral exponents can also be computed directly from the
cusp ranks and the coefficients of $\psi$
(Corollary~\ref{cor.realhyperbolicselfjoinings}).
In particular, self-joinings of two finite-area Fuchsian
representations that are not conjugate in
$\operatorname{PGL}(2,\R)$ give examples in which $\nu_\psi$
is locally doubling and locally reverse doubling but
$\cal H_\psi(\La_\theta)=0$
(Example~\ref{ex.cuspedsurfaceselfjoining}).

\medskip

\noindent{\bf AI disclosure:}
The argument establishing the necessity of the condition in
Theorem~\ref{thm.Hmeasure} was suggested by an AI tool.

\section{Preliminaries}
Let $G$ be a connected semisimple real algebraic group. 
Let $P$ be a minimal parabolic subgroup with a fixed Langlands decomposition $P=MAN$ where $A$ is a maximal real split torus of $G$, $M$ is the maximal compact subgroup of $P$ commuting with $A$ and $N$ is the unipotent radical of $P$.
Let $\fg$ and $\fa$ denote, respectively, the Lie algebras of $G$
and $A$. Fix a positive Weyl chamber $\fa^+\subset \fa$ so that
$\log N$ consists of positive root subspaces and
set $A^+=\exp \fa^+$. We fix a maximal compact subgroup $K< G$ such that the Cartan decomposition $G=K A^+ K$ holds. We denote by 
$\mu : G \to \fa^+$ the Cartan projection defined by the condition $g\in K\exp \mu(g) K$ for $g \in G$. Let $X = G/K$ be the associated Riemannian symmetric space and $o=[K]\in X$.  Fix an $\text{Ad}(K)$-invariant inner product on $\fg$ with associated norm $\|\cdot \|$. This induces the left $G$-invariant Riemannian metric $d$ on $X$.

\begin{lemma} \cite[Lemma 4.6]{Benoist1997proprietes} \label{lem.cptcartan}
For any compact subset $Q \subset G$, there exists $C=C(Q)>0$ such that for all $g \in G$, $$\sup_{q_1, q_2\in Q} \| \mu(q_1gq_2) -\mu(g)\| \le C .$$  
\end{lemma}

Let $\Pi $ denote the set of simple roots determined by $\fa^+$.
We fix a non-empty subset $\theta\subset \Pi$. 
Let $P_\theta$ denote the corresponding standard parabolic subgroup with the convention that $P_\Pi=P$, and set $\F_\theta:=G/P_\theta$.
We also set
\begin{equation*}
\mathfrak{a}_\theta :=\bigcap_{\alpha \in \Pi-\theta} \ker \alpha  \quad \text{and} \quad \fa_\theta^+  :=\fa_\theta\cap \fa^+. \end{equation*}
Let
\be \label{att}
p_\theta:\mathfrak{a}\to\mathfrak{a}_\theta
\ee denote the projection invariant under all Weyl elements fixing $\fa_\theta$ pointwise. We write $\mu_{\theta} := p_{\theta} \circ \mu :G\to \fa_\theta^+.$
We identify $\fa_\theta^*=\op{Hom}(\fa_\theta, \br)$  with the subspace of $\fa^*$ consisting of linear forms invariant under $p_\theta$.

 Abusing notation, for $p, q \in X$, we set
 $$
 \mu(p) := \mu(g) \quad \text{and} \quad \mu(p^{-1}q) := \mu(g^{-1} h)
 $$
 for $g, h \in G$ such that $go = p$ and $ho = q$. This definition is independent of the choice of $g$ and $h$, and we similarly define $\mu_{\theta}(p)$ and $\mu_{\theta}(p^{-1}q)$.

Let $w_0\in K$ represent the longest Weyl element.  The opposition
involution is
\[
    \i:=-\op{Ad}_{w_0}:\fa\to\fa,
\]
and it induces an involution of $\Pi$, again denoted by $\i$.

\subsection*{Limit set} 
The subgroup $K$ acts transitively on $\F_\theta$, and hence
 $\F_\theta\simeq K/ M_\theta$ where $M_\theta:=P_\theta\cap K$.  Set $\xi_\theta:=[M_\theta]\in \F_\theta$.
\begin{definition} \label{fc} For a sequence $g_i\in G$  and $\xi\in \ft$, we write $\lim_{i\to \infty} g_i =\xi$ and
 say $g_i $ {\it converges} to $\xi$ if \begin{itemize}
     \item for each $\alpha\in \theta$, $\alpha(\mu(g_i)) \to \infty$ as $i\to \infty$; 
\item $\lim_{i\to\infty} \kappa_{i}\xi_\theta= \xi$ in $\F_\theta$ for some $\kappa_{i}\in K$ such that $g_i\in \kappa_{i} A^+ K$.
 \end{itemize}         
\end{definition}

 The {\it $\theta$-limit set} of a discrete subgroup $\Ga$ can be defined as follows:
\be \label{def.limitset} \lat=\lat(\Ga):=\{\lim {\ga}_i\in \F_\theta: {\ga}_i\in \Ga\}\ee  where $\lim \ga_i$ is defined as in Definition \ref{fc}.
If $\Ga$ is Zariski dense, this is the unique $\Ga$-minimal subset of $\F_{\theta}$ \cite{Benoist1997proprietes}.

\subsection*{Busemann map}
The {\it $\frak a$-valued Busemann map} $\beta: \cal F_\Pi \times G \times G \to\frak a $ is defined as follows: for $\xi\in \cal F_\Pi$ and $g, h\in G$,
$$  \beta_\xi ( g, h):=\sigma (g^{-1}, \xi)-\sigma(h^{-1}, \xi)$$
where  $\sigma(g^{-1},\xi)\in \fa$ 
is the unique element such that $g^{-1}k \in K \exp (\sigma(g^{-1}, \xi)) N$ for any $k\in K$ with $\xi=kP$.
For $(\xi,g,h)\in \cal F_\theta\times G\times G$, we define
 \be\label{Bu} \beta_{\xi}^\theta (g, h): = 
p_\theta ( \beta_{\xi_0} (g, h)) \ee 
for any $\xi_0\in \F_\Pi$ projecting to $\xi$.
This is well-defined independent of the choice of $\xi_0$ 
\cite[Lemma 6.1]{Quint2002Mesures}. Moreover, since product map $ K\times A \times N\to G$ is a diffeomorphism, the Busemann maps are continuous.

\subsection*{Gromov product}

Two points $\xi \in \F_{\theta}$ and $\eta \in \F_{\i(\theta)}$ are said to be {\it transverse}, or  {\it in general position} if 
\be \label{gp} \xi=gP_{\theta}\quad \text{ and } \quad \eta = gw_0 P_{\i(\theta)} \quad \text{ for some $g\in G$}.\ee 
We set
\be\label{fgp} \F_\theta^{(2)}=\{(\xi,\eta)\in \F_\theta\times \F_{\i(\theta)}:  \text{$\xi, \eta$ are in general position} \}\ee 
which is the unique open $G$-orbit in $\F_\theta\times \F_{\i(\theta)}$ under the diagonal $G$-action.

For $(\xi, \eta) \in \F_{\theta}^{(2)}$, we define the {\it $\fa_{\theta}$-valued Gromov product } as 
\be \label{eqn.Gromovproduct}
\cal G^{\theta} (\xi, \eta) = \frac{1}{2} \left( \beta_{\xi}^{\theta}(e, g) + \i \beta_{\eta}^{\i(\theta)}(e, g) \right)
\ee
where $g \in G$ satisfies $(gP_\theta, gw_0P_{\i(\theta)}) = (\xi, \eta)$. This is independent of the choice of $g$ \cite[Lemma 9.13]{KOW_indicators}.

\subsection*{Patterson--Sullivan measures}
For $\psi\in \fa_\theta^*$, a {\it $(\Gamma, \psi)$-conformal measure} is a Borel probability measure on $\F_\theta$ such that 
\be \label{eqn.psmeas}
\frac{d \gamma_*\nu}{d\nu}(\xi)=e^{\psi(\beta_\xi^\theta(e,\gamma))} \quad \text{for all $\gamma \in \Gamma$ and $\nu$-a.e.$ \xi \in \F_\theta$}
\ee
 where ${\ga}_* \nu(D) = \nu(\ga^{-1}D)$ for any Borel subset $D\subset \F_\theta$ and $\beta_\xi^\theta$ denotes the $\fa_\theta$-valued Busemann map defined in \eqref{Bu}. A $(\Ga, \psi)$-conformal measure supported on $\La_\theta$ is called a {\it $(\Ga, \psi)$-Patterson--Sullivan measure}.

\section{Relatively hyperbolic groups and Gromov models} \label{sec.relhyp}

In this section, we recall relatively hyperbolic groups, Gromov hyperbolic spaces,
and Gromov models.

\subsection*{Relatively hyperbolic groups}

Let $\Ga$ be a countable group acting by homeomorphisms on a compact metrizable
space $\cal X$.  The action is called a {\it convergence group action} if, for
every sequence of distinct elements $\ga_n\in\Ga$, there exist a subsequence
$\ga_{n_k}$ and points $a,b\in\cal X$ such that $\ga_{n_k}(x)$ converges to
$a$ for all $x\in\cal X-\{b\}$, uniformly on compact subsets.

An infinite-order element $\ga\in\Ga$ is called {\it loxodromic} if it fixes
exactly two points of $\cal X$, and {\it parabolic} if it fixes exactly one
point.  An infinite subgroup $\mathsf P<\Ga$ is called {\it parabolic} if it
fixes a point of $\cal X$ and every infinite-order element of $\mathsf P$ is
parabolic.

A point $\xi\in\cal X$ is called a {\it conical limit point} if there exist a
sequence of distinct elements $\ga_n\in\Ga$ and distinct points $a,b\in\cal X$
such that
\[
    \ga_n^{-1}\xi\to a
    \quad\text{and}\quad
    \ga_n^{-1}\eta\to b
    \quad\text{for all } \eta\in\cal X-\{\xi\}.
\]
A point $\xi\in\cal X$ is called a {\it parabolic limit point} if it is fixed
by a parabolic subgroup of $\Ga$.  Such a point is called {\it bounded
parabolic} if
\[
    \stab_{\Ga}(\xi)\backslash(\cal X-\{\xi\})
\]
is compact.  The action of $\Ga$ on $\cal X$ is called a {\it geometrically
finite convergence group action} if every point of $\cal X$ is either conical
or bounded parabolic.  A typical example is the action of a geometrically finite
Kleinian group on its limit set.

Let $\Ga$ be a finitely generated group and let $\cal P$ be a finite collection
of finitely generated infinite subgroups of $\Ga$.  We say that $\Ga$ is
{\it hyperbolic relative to $\cal P$}, or that $(\Ga,\cal P)$ is
{\it relatively hyperbolic}, if $\Ga$ admits a geometrically finite convergence
group action on a compact perfect metrizable space $\cal X$ whose maximal
parabolic subgroups are precisely
\[
    \cal P^\Ga
    :=
    \{\ga\mathsf P\ga^{-1}:\mathsf P\in\cal P,\ \ga\in\Ga\}.
\]

Bowditch \cite{Bowditch_relhyp} showed that, if $\Ga$ is hyperbolic relative to
$\cal P$, then the space $\cal X$ satisfying the above condition is unique up to
$\Ga$-equivariant homeomorphism.  This space is called the {\it Bowditch
boundary} and is denoted by $\partial(\Ga,\cal P)$.  Since
$\partial(\Ga,\cal P)$ is assumed to be perfect, we have
$\#\partial(\Ga,\cal P)\ge 3$; equivalently, $(\Ga,\cal P)$ is non-elementary.

\subsection*{Gromov hyperbolic spaces}

A proper geodesic metric space $(Y,d)$ is called {\it Gromov hyperbolic} if
there exists $\delta>0$ such that every geodesic triangle in $Y$ is
$\delta$-thin; that is, each side is contained in the $\delta$-neighborhood of
the union of the other two sides.  The Gromov boundary $\partial Y$ is the set
of equivalence classes of geodesic rays, where two rays are equivalent if they
have finite Hausdorff distance.  We write
\[
    \overline Y := Y\cup\partial Y
\]
for the corresponding compactification.

For $C_1,C_2\ge 1$ and an interval $I\subset\R$, a map $\sigma:I\to Y$ is
called a $(C_1,C_2)$-quasi-geodesic if
\[
    C_1^{-1}|t-s|-C_2
    \le d(\sigma(t),\sigma(s))
    \le C_1|t-s|+C_2
    \quad\text{for all } t,s\in I.
\]
We also call the image $\sigma(I)$ a $(C_1,C_2)$-quasi-geodesic.  We shall use
the following standard stability property.

\begin{lemma} \label{lem.stableqg}
For any $C_1,C_2\ge 1$, there exists $R>0$ such that any two
$(C_1,C_2)$-quasi-geodesics in $Y$ with the same endpoints in $\overline Y$
have Hausdorff distance at most $R$.
\end{lemma}

We use the following notation. 

\begin{Def} \label{def.shadowdefinition}
Let $y,y_1,y_2\in Y$, let $z_1,z_2\in\overline Y$, and let $R>0$.
\begin{enumerate}
    \item Let
    \[
        B(y,R)=\{x\in Y:d(y,x)<R\}
    \]
    denote the ball of radius $R$ centered at $y$.

    \item The notation $[y_1,y_2]$ denotes a choice of geodesic in $Y$
    connecting $y_1$ to $y_2$.  Such a geodesic need not be unique.

    \item The shadow $O_R(y_1,y_2) \subset \partial Y$ is
    \[
        O_R(y_1,y_2)
        =
        \{\xi\in\partial Y:
        \text{ some geodesic ray } [y_1,\xi] \text{ intersects } B(y_2,R)\}.
    \]

    \item We denote by $\pi_{z_1,z_2}(y)$ the set of nearest-point projections
    of $y$ to all geodesics between $z_1$ and $z_2$.

    \item We denote by $\pi_{[z_1,z_2]}(y)$ the set of nearest-point
    projections of $y$ to the chosen geodesic $[z_1,z_2]$.
\end{enumerate}
\end{Def}

It is well known that, for any $x\in\pi_{z_1,z_2}(y)$, the concatenation
$[y,x]\cup[x,z_1]$ is a $(1,O(\delta))$-quasi-geodesic, where $O(\delta)$
denotes a constant depending only on $\delta$.

Given $o_Y\in Y$ and two distinct points $\xi,\eta\in Y\cup\partial Y$, define
their Gromov product with respect to $o_Y$ by
\[
\langle \xi,\eta\rangle_{o_Y}
:=
\sup \liminf_{i,j\to\infty}
\frac{1}{2}\bigl(d(o_Y,x_i)+d(o_Y,y_j)-d(x_i,y_j)\bigr),
\]
where the supremum is taken over all sequences $x_i,y_j\in Y$ such that
$x_i\to\xi$ and $y_j\to\eta$.  This quantity measures the distance from $o_Y$
to $\pi_{\xi,\eta}(o_Y)$ up to a uniform additive error depending only on
$\delta$.  We shall use the standard inequality
\begin{equation} \label{eqn.Gromovineq}
\langle \xi,\eta\rangle_{o_Y}
\ge
\min\left(
    \langle \xi,\zeta\rangle_{o_Y},
    \langle \zeta,\eta\rangle_{o_Y}
\right)
-
O(\delta)
\end{equation}
for all $\xi,\eta,\zeta\in Y\cup\partial Y$.

The following standard comparison between shadows and Gromov products will be
used repeatedly.

\begin{lemma} \label{lem.shadowandgromprod}
Let $o_Y\in Y$, $\xi\in\partial Y$, and $\xi_t\in[o_Y,\xi]$ be such
that $d(o_Y,\xi_t)=t$ for $t\ge 0$.  For any $R>0$:
\begin{enumerate}
    \item if $\eta\in O_R(o_Y,\xi_t)$, then
    \[
        \langle \xi,\eta\rangle_{o_Y}
        \ge t-R-O(\delta);
    \]
    \item if $\langle \xi,\eta\rangle_{o_Y}
        \ge t-R+O(\delta)$,
    then $$\eta\in O_R(o_Y,\xi_t).$$
\end{enumerate}
\end{lemma}

\begin{proof}
We first prove (1).
Let $\eta\in O_R(o_Y,\xi_t)$ and let $y\in\pi_{[\xi,\eta]}(o_Y)$.  It
suffices to consider the case $\langle \xi,\eta\rangle_{o_Y}\le t$.  In this
case, $\xi_t$ is $O(\delta)$-close to $[\xi,\eta]$.  Choose
$x\in[\xi,\eta]$ with $d(\xi_t,x)<O(\delta)$.  Since
$\eta\in O_R(o_Y,\xi_t)$, Lemma \ref{lem.stableqg} implies that $\xi_t$ is
$R+O(\delta)$-close to $[o_Y,y]\cup[y,\eta]$.  Hence $x$ is also
$R+O(\delta)$-close to $[o_Y,y]\cup[y,\eta]$.

If $x$ is $R+O(\delta)$-close to $[y,\eta]$, then
$d(x,y)<R+O(\delta)$, since both $x$ and $[y,\eta]$ lie on the geodesic
$[\xi,\eta]$.  Therefore
\[
    d(o_Y,y)
    \ge d(o_Y,x)-R-O(\delta)
    \ge d(o_Y,\xi_t)-R-O(\delta).
\]
Since $d(o_Y,\xi_t)=t$ and
$d(o_Y,y)\le \langle \xi,\eta\rangle_{o_Y}+O(\delta)$, this proves the desired
estimate in this case.

If $x$ is $R+O(\delta)$-close to $[o_Y,y]$, then choose
$w\in \pi_{[o_Y,y]}(x)$. Then we have $d(x,w)<R+O(\delta)$.  On the other hand, the concatenation
$[w,y]\cup[y,x]$ is a $(1,O(\delta))$-quasi-geodesic.  Hence $y$ is also
$R+O(\delta)$-close to $x$, and the same argument as above proves (1).

For (2), suppose that
\[
    \langle \xi,\eta\rangle_{o_Y}\ge t-R+O(\delta).
\]
Then $\xi_{t-R+O(\delta)}$ is $O(\delta)$-close to $[o_Y,\eta]$.  Hence
$\xi_t$ is $R$-close to $[o_Y,\eta]$, and therefore
$\eta\in O_R(o_Y,\xi_t)$.
\end{proof}

\subsection*{Gromov models}

For the rest of this section, let $(\Ga,\cal P)$ be a relatively hyperbolic
group.  Following \cite{KL_relativizing}, we recall the basic properties of
Gromov models.  Such a model is a proper geodesic Gromov hyperbolic space on
which $\Ga$ acts in a way analogous to the action of a geometrically finite
Kleinian group on the real hyperbolic space.

Let $(Y,d)$ be a proper geodesic Gromov hyperbolic space.  For any
$\eta\in\partial Y$, a horofunction based at $\eta$ is obtained as follows:
if $y_n \in Y$ is a sequence converging to $\eta$, then, after passing to a subsequence, there exist
$t_n\to\infty$ and a function $h:Y\to\R$ such that
\[
    h(x)=\lim_{n\to\infty} d(x,y_n)-t_n ,
\]
uniformly on compact subsets. There exists a constant $C>0$, depending only
on $Y$, such that every such function satisfies
\begin{equation}\label{hii}
\left|
\bigl(h(\ell_{\eta}(t_2))-h(\ell_{\eta}(t_1))\bigr)
-
(t_1-t_2)
\right|
\le C
\quad\text{for all } t_1,t_2\ge 0,
\end{equation}
for every unit-speed geodesic ray $\ell_{\eta}:[0,\infty)\to Y$ asymptotic to
$\eta$.  By a horofunction at $\eta$, we mean any function satisfying
\eqref{hii} for every such ray.

A subset $H\subset Y$ is called a horoball at $\eta$ if there exists a
horofunction $h$ at $\eta$ such that
\[
    \{h\le 0\}\subset H\subset \{h\le 10C\}.
\]
Then $\overline H\cap\partial Y=\{\eta\}$, and horoballs are uniformly
quasi-convex.

\begin{Def} \label{def.gromovmodel}
A proper geodesic Gromov hyperbolic space $(Y,d)$ is called a {\it Gromov model}
for $(\Ga,\cal P)$ if:
\begin{enumerate}
    \item $\Ga$ acts properly discontinuously on $Y$ by isometries;

    \item $Y$ is taut, meaning that there exists $R>0$ such that every point
    of $Y$ lies within distance $R$ of a bi-infinite geodesic;

    \item there exists a $\Ga$-invariant collection $\cal B=\{B_i\}$ of
    disjoint open horoballs such that the stabilizer of each $B_i$ in $\Ga$ is
    of the form $\ga\mathsf P\ga^{-1}$ for some $\ga\in\Ga$ and
    $\mathsf P\in\cal P$;

    \item the action of $\Ga$ on $Y-\bigcup_i B_i$ is cocompact.
\end{enumerate}
\end{Def}

By the uniqueness of the Bowditch boundary, the Gromov boundary of a Gromov
model $Y$ for $(\Ga,\cal P)$ is $\Ga$-equivariantly homeomorphic to
$\partial(\Ga,\cal P)$.  By \cite[Proposition 2.12]{Bowditch1999convergence},
the action of $\Ga$ on $\partial Y$ is a convergence group action.

For each $B\in\cal B$, write
\[
    \{p\}=\overline B\cap\partial Y.
\]
We call $p$ the basepoint of $B$ and also write $B=B_p$.  Given
$\mathsf P\in\cal P^\Ga$, we denote its basepoint by
$\xi_{\mathsf P}\in\partial Y$.

These basepoints are precisely the parabolic limit points of $(\Ga,\cal P)$ in
$\partial Y$.  We shall use the following standard properties of parabolic
subgroups, whose proofs are included for completeness.

\begin{lemma} \label{lem.shadowincpt}
Let $o_Y\in Y$, let $\mathsf P\in\cal P^\Ga$, and set
$\xi=\xi_{\mathsf P}$.  For every $R>0$, there exists a compact set
$Q\subset\partial Y-\{\xi\}$ such that
\[
    O_R(o_Y,g o_Y)\subset gQ
    \quad\text{for all but finitely many } g\in\mathsf P.
\]
\end{lemma}

\begin{proof}
Fix $R>0$ and suppose the claim fails.  Let
$Q_n\subset\partial Y-\{\xi\}$ be an increasing sequence of compact
sets with
\[
    \bigcup_n Q_n=\partial Y-\{\xi\}.
\]
Then there exist sequences $g_n\in\mathsf P$ and
$\eta_n\in O_R(o_Y,g_n o_Y)$ such that $\eta_n\notin g_nQ_n$. We may assume that the sequence
$g_n$ is infinite. Since $g_n^{-1}\eta_n\notin Q_n$ for every $n$, 
$g_n^{-1}\eta_n\to\xi$.  On the other hand,
\[
    g_n^{-1}\eta_n\in O_R(g_n^{-1}o_Y,o_Y)
    \quad\text{and}\quad
    g_n^{-1}o_Y\to\xi,
\]
which is impossible.  This proves the lemma.
\end{proof}

\begin{lemma} \label{lem.cptinshadow}
Let $o_Y\in Y$, let $\mathsf P\in\cal P^\Ga$, and set
$\xi=\xi_{\mathsf P}$.  For every compact set
$Q\subset\partial Y-\{\xi\}$, there exists $R>0$ such that
\[
    gQ\subset O_R(o_Y,g o_Y)
    \quad\text{for all } g\in\mathsf P.
\]
\end{lemma}

\begin{proof}
Let $Q\subset\partial Y-\{\xi\}$ be compact and suppose the conclusion
fails.  Then, for every $n\ge 1$, there exists $g_n\in\mathsf P$ such that
\[
    g_nQ\not\subset O_n(o_Y,g_n o_Y).
\]
Thus $g_n$ is an infinite sequence and
\[
    Q\not\subset O_n(g_n^{-1}o_Y,o_Y)
    \quad\text{for every } n\ge 1.
\]
Since $g_n^{-1}o_Y\to\xi$, this forces $\xi\in Q$, a contradiction.
\end{proof}

Finally, a point $\xi\in\partial Y$ is conical if and only if, for any
$o_Y\in Y$, there exist $R>0$ and an infinite sequence $\ga_n\in\Ga$ such that
\[
    \xi\in O_R(o_Y,\ga_n o_Y)
    \quad\text{for all } n.
\]
In this case, we say that $\ga_n o_Y$ converges conically to $\xi$. It is easy to see that this notion of conicality is equivalent to the conicality defined in terms of convergence action.

\section{Relatively Morse groups} \label{sec.relmorse}

Let $\Ga<G$ be a discrete subgroup which is hyperbolic relative to a finite
collection $\cal P$ of finitely generated infinite subgroups of $\Ga$.  We fix
a non-empty subset $\theta\subset\Pi$.

\begin{Def} \label{def.Morsesubgroup}
We say that $\Ga$ is {\it $\theta$-Morse relative to $\cal P$} if there
exist a Gromov model $Y$ for $(\Ga,\cal P)$ and a $\Ga$-equivariant
quasi-isometric embedding
\[
    f:Y\to X
\]
such that, if $\L_f\subset\fa^+$ denotes the asymptotic cone of
\[
    \{\mu(f(x)^{-1}f(y)):x,y\in Y\},
\]
then
\[
    \L_f\cap\ker\alpha=\{0\}
    \quad\text{for every } \alpha\in\theta .
\]
Such a map $f$ is called a Morse embedding of $Y$.
\end{Def}

For a Morse embedding $f:Y\to X$, we call $\L_f$ the Morse limit cone of
$f$.  We also set
\[
    \L_{\theta,f}:=p_\theta(\L_f)
\]
and call it the Morse $\theta$-limit cone.  Since $\L_f=\i(\L_f)$, being
$\theta$-Morse relative to $\cal P$ is equivalent to being
$\theta\cup\i(\theta)$-Morse relative to $\cal P$.  Thus, without loss of
generality, we assume
\[
    \theta=\i(\theta)
\]
throughout the rest of this section.

A subgroup $\Ga<G$ which is hyperbolic relative to $\cal P$ is called
{\it $\theta$-Anosov relative to $\cal P$} if:
\begin{itemize}
    \item $\Ga$ is $\theta$-regular, meaning that
    \[
        \min_{\alpha\in\theta}\alpha(\mu(\ga_n))\to\infty
    \]
    for every infinite sequence $\ga_n\in\Ga$;

    \item there exists a transverse $\Ga$-equivariant embedding
    \[
        \zeta:\partial Y\to\F_\theta,
    \]
    i.e., it sends distinct points of $\partial Y$ to points in general position.
\end{itemize}
If $\Ga$ is $\theta$-Morse relative to $\cal P$, then $\Ga$ is
$\theta$-Anosov relative to $\cal P$.  Indeed, $\theta$-regularity follows
from the fact that $f$ is a quasi-isometric embedding and
\[
    \L_f\cap\bigcup_{\alpha\in\theta}\ker\alpha=\{0\}.
\]
Moreover, by \cite[Theorem 1.4]{KLP_2018}, the Morse embedding
$f:Y\to X$ extends continuously to a transverse $\Ga$-equivariant embedding
\[
    f:\partial Y\to\F_\theta.
\]
Its
image is the $\theta$-limit set $\La_\theta$.

\begin{remark}
A map $f:Y\to X$ satisfying
\[
    \L_f\cap\bigcup_{\alpha\in\theta}\ker\alpha=\{0\}
\]
is called uniformly $\theta$-regular in \cite{KL_relativizing}.  Although
this condition is a priori different from the original definition of a
Morse embedding in \cite{KL_relativizing}, the Morse lemma
\cite{KLP_2018} implies that the two notions are equivalent for
quasi-isometric embeddings.
\end{remark}

\subsection*{Geometric properties of relatively Morse groups}

For the rest of this section, let $\Ga<G$ be $\theta$-Morse relative to
$\cal P$, with Morse embedding $f:Y\to X$ from a Gromov model $(Y,d)$ and
continuous extension $f:\partial Y\to\F_\theta$.  We fix a basepoint
$o_Y\in Y$ and  we may assume that $f(o_Y)=o$.

The following proposition is a key ingredient in the proof of the global
shadow lemma.  Its second assertion follows from Lemma \ref{lem.gromovproductnearest} below.

\begin{proposition} \label{prop.gromovproductalongparabolic}
Let $\mathsf P\in\cal P$ and let $p=\xi_{\mathsf P}$.  For every compact
subset $Q\subset\partial Y-\{p\}$, there exists a constant
$c=c(Q)>0$ such that
\[
    \sup_{\xi\in Q,\,\ga\in\mathsf P}
    \left\|
        \cal G^\theta(f(p),\ga f(\xi))
        -
        \frac{1}{2}\mu_\theta(\ga)
    \right\|
    \le c .
\]
In particular, $\mu_\theta(\ga)-\i\mu_\theta(\ga)$ is uniformly bounded
for all $\ga\in\mathsf P$.
\end{proposition}

We first collect several lemmas.  The following was proved in
\cite[Lemma 6.6]{LO_invariant} when $\Ga$ is Borel Anosov, using the
Morse property of the orbit map of an Anosov group into $X$ \cite{KLP_2018}.
The same argument applies in our setting, replacing the Cayley graph of a
hyperbolic group by the Gromov hyperbolic space $Y$.

\begin{lemma} \label{lem.gromovproductnearest}
There exists $C>0$ such that, for any distinct $\xi,\eta\in\partial Y$,
\[
    \sup_{z\in\pi_{\xi,\eta}(o_Y)}
    \left\|
        \cal G^\theta(f(\xi),f(\eta))
        -
        \frac{1}{2}
        \left(
            \mu_\theta(f(z))+\i\mu_\theta(f(z))
        \right)
    \right\|
    \le C .
\]
\end{lemma}

The following was proved in \cite[Corollary 4.12]{DKO_AR}.  Although the
original statement applies to the values of a linear form on the Cartan projections, the
same proof gives the following vector-valued version.

\begin{lemma} \label{lem.DKO_additive}
There exists $D_0>0$ such that, for any $x,z\in Y$ and any
$y\in[x,z]$,
\[
\left\|
    \mu_\theta(f(x)^{-1}f(z))
    -
    \left(
        \mu_\theta(f(x)^{-1}f(y))
        +
        \mu_\theta(f(y)^{-1}f(z))
    \right)
\right\|
< D_0 .
\]
\end{lemma}

We also need the following elementary observation.

\begin{lemma} \label{lem.projinhoroball}
Let $\mathsf P\in\cal P$ and let $p=\xi_{\mathsf P}$.  For every
$R>0$, there exists a finite subset $\mathsf P(R)\subset\mathsf P$ such
that, for every $\ga\in\mathsf P-\mathsf P(R)$, every geodesic
$[\xi,p]$ with $\xi\in\partial Y-\{p\}$, and every
$u\in B(o_Y,R)\cap[\xi,p]$, we have
\[
    \pi_{\ga[\xi,p]}(u)
    \subset
    \ga[u,p]
\]
where we choose $[u, p] \subset [\xi, p]$.
\end{lemma}

\begin{proof}
Suppose not.  Then there exist $R>0$, an infinite sequence
$\ga_n\in\mathsf P$, and sequences
\[
    \xi_n\in\partial Y - \{p\}, \quad
    u_n\in B(o_Y,R)\cap[\xi_n,p],
    \quad \text{and} \quad
    y_n\in\pi_{\ga_n[\xi_n,p]}(u_n),
\]
such that
\[
    y_n\notin\ga_n[u_n,p]
\]
where we choose $[u_n, p] \subset [\xi_n, p]$, for all $n \ge 1$.

For each $n$, choose  $[y_n,p] \subset \ga_n[\xi_n,p]$.  By the assumption, $\ga_nu_n\in[y_n,p]$.  By Gromov
hyperbolicity, the concatenation
\[
    [u_n,y_n]\cup[y_n,p]
\]
is a $(1,O(\delta))$-quasi-geodesic.  Hence
\[
    p\in O_{O(\delta)}(u_n,\ga_nu_n).
\]
Since $d(o_Y,u_n)<R$, it follows that for all $n$, 
\[
    p\in O_{R'}(o_Y,\ga_no_Y)
\]
for some $R'>0$ depending only on $R$ and $\delta$.  Thus
$\ga_no_Y$ converges conically to $p$, contradicting the fact that $p$ is
a parabolic limit point.
\end{proof}

We will use the following coarse midpoint estimate in the Gromov model:

\begin{lemma} 
\cite[Lemma 3.8]{BT_global} \label{lem.braytiozzo}
Let $\mathsf P\in\cal P$ and let $p=\xi_{\mathsf P}$.  Let $R>0$, and
let $Q\subset\partial Y-\{p\}$ be compact.  Then there exists
$D>0$ such that, for every $x\in B(o_Y,R)$ and every $\ga\in\mathsf P$,
\[
    \sup_{\xi\in Q,\, y\in\pi_{\ga\xi,p}(x)}
    \left|
        d(x,y)-\frac{1}{2}d(x,\ga x)
    \right|
    \le D .
\]
\end{lemma}

As an intermediate step in the proof of Proposition
\ref{prop.gromovproductalongparabolic}, we prove the following.  We write
$\approx$ for equality up to a uniform additive error.

\begin{lemma} \label{lem.parabolictransmiddle}
Let $\mathsf P\in\cal P$ and let $p=\xi_{\mathsf P}$.  Let $R>0$, and
let $Q\subset\partial Y-\{p\}$ be compact.  Then there exists a
constant $C=C(R,Q)>0$ such that, for every $\ga\in\mathsf P$, every
geodesic $[\xi,p]$ with $\xi\in Q$, every
$x\in[\xi,p]\cap B(o_Y,R)$, and every $w\in\ga^{-1}\pi_{\ga[\xi,p]}(x)$,
there exists $y\in[x,\ga x]$ such that
\[
    d(y,w)<C,
    \quad
    d(y,\ga w)<C,
    \quad \text{and} \quad 
    d(w,\ga w)<C.
\]
\end{lemma}

\begin{proof}
Let $\mathsf P(R)\subset\mathsf P$ be the finite subset given by Lemma
\ref{lem.projinhoroball}.  It suffices to consider
$\ga\in\mathsf P-\mathsf P(R)$. Let $\xi\in Q$, and choose a geodesic $[\xi,p]$.  Let
$x\in[\xi,p]\cap B(o_Y,R)$, and let $w\in[\xi,p]$ be such that
$\ga w\in\pi_{\ga[\xi,p]}(x)$.
By Lemma \ref{lem.braytiozzo},
\begin{equation} \label{egw}
    d(x,\ga w)
    \approx
    \frac{1}{2}d(x,\ga x),
\end{equation}
with additive error depending only on $R$ and $Q$.

Consider the geodesic $[\ga x,p]\subset\ga[\xi,p]$ and the geodesic
triangle
\[
    [x,p]\cup[x,\ga x]\cup[\ga x,p].
\]
By Lemma \ref{lem.projinhoroball}, we have $\ga w\in[\ga x,p]$.  Since
$w\in[\xi,p]$, we also have $w\in[x,p]$.  Since
$\ga w\in\pi_{\ga[\xi,p]}(x)$, Gromov hyperbolicity implies that there
exist $y\in[x,\ga x]$ and $z\in[x,p]$ such that the three points
$\ga w$, $y$, and $z$ are uniformly close to one another.  Together with
\eqref{egw}, this gives
\[
    \frac{1}{2}d(x,\ga x)
    \approx
    d(x,\ga w)
    \approx
    d(x,y)
    \approx
    d(x,z).
\]
Hence
\begin{equation} \label{xgx}
    d(x,\ga x)
    \approx
    d(x,y)+d(x,z).
\end{equation}
On the other hand,
\[
\begin{aligned}
    d(x,\ga x)
    &=
    d(x,y)+d(y,\ga x) \\
    &\approx
    d(x,y)+d(\ga w,\ga x) \\
    &=
    d(x,y)+d(w,x).
\end{aligned}
\]
Comparing this with \eqref{xgx}, we obtain
\[
    d(x,w)\approx d(x,z).
\]
Since $z,w\in[x,p]$, this implies that the points $z$ and $w$ are uniformly close.  Since
$\ga w$, $y$, and $z$ are uniformly close, it follows that $\ga w$, $y$,
and $w$ are uniformly close as well.  This proves the lemma.
\end{proof}

\subsection*{Proof of Proposition \ref{prop.gromovproductalongparabolic}}

There exists $R=R(Q)>0$ such that every geodesic $[\xi,p]$ with
$\xi\in Q$ intersects $B(o_Y,R)$. 
Let $\ga \in \mathsf{P}$.

Let $\xi\in Q$, and choose a geodesic $[\xi,p]$.  Choose
$x\in[\xi,p]\cap B(o_Y,R)$, and let $w\in[\xi,p]$ be such that
\[
    \ga w\in\pi_{\ga[\xi,p]}(x).
\]
Since $x\in B(o_Y,R)$, the sets $\pi_{\ga[\xi,p]}(x)$ and
$\pi_{\ga[\xi,p]}(o_Y)$ have uniformly bounded Hausdorff distance.  Hence,
by Lemmas \ref{lem.cptcartan} and \ref{lem.gromovproductnearest},
\[
    \cal G^\theta(f(p),\ga f(\xi))
    \approx
    \frac{1}{2}
    \left(
        \mu_\theta(f(\ga w))
        +
        \i\mu_\theta(f(\ga w))
    \right).
\]

By Lemma \ref{lem.parabolictransmiddle}, there exists
$y\in[x,\ga x]$ such that $y$, $w$, and $\ga w$ are uniformly close.  Since
$y\in[x,\ga x]$, Lemma \ref{lem.DKO_additive} gives
\[
    \mu_\theta(f(x)^{-1}f(\ga x))
    \approx
    \mu_\theta(f(x)^{-1}f(y))
    +
    \mu_\theta(f(y)^{-1}f(\ga x)).
\]
Since $f$ is a quasi-isometric embedding, the points
$f(y)$, $f(w)$, and $f(\ga w)$ are uniformly close.  Since
$x\in B(o_Y,R)$, the point $f(x)$ remains in a uniformly bounded subset of
$X$.  Therefore, by Lemma \ref{lem.cptcartan},
\[
\begin{aligned}
    \mu_\theta(\ga)
    &\approx
    \mu_\theta(f(x)^{-1}f(\ga x)) \\
    &\approx
    \mu_\theta(f(x)^{-1}f(\ga w))
    +
    \mu_\theta(f(w)^{-1}f(x)) \\
    &\approx
    \mu_\theta(f(x)^{-1}f(\ga w))
    +
    \i\mu_\theta(f(x)^{-1}f(w)) \\
    &\approx
    \mu_\theta(f(\ga w))
    +
    \i\mu_\theta(f(\ga w)) \\
    &\approx
    2\cal G^\theta(f(p),\ga f(\xi)).
\end{aligned}
\]
This finishes the proof of the main claim.

For the second assertion, note that  by Lemma
\ref{lem.gromovproductnearest},
\[
    \cal G^\theta(f(p),\gamma f(\xi))
    \approx
    \frac12\left(
        \mu_\theta(f(z_\gamma))
        +
        \i\mu_\theta(f(z_\gamma))
    \right)
\]
for some $z_\gamma\in\pi_{p,\gamma\xi}(o_Y)$.  Hence
\[
    \mu_\theta(\gamma)
    \approx
    \mu_\theta(f(z_\gamma))
    +
    \i\mu_\theta(f(z_\gamma)).
\]
Since the right-hand side is $\i$-invariant, this implies that $\mu_\theta(\gamma)-\i\mu_\theta(\gamma)$ is uniformly bounded.
\qed

\subsection*{A metric-like function on $Y$}

We identify $\fa_\theta^*$ with the subspace of $\fa^*$ obtained by
precomposing with the projection $p_\theta:\fa\to\fa_\theta$.  Let
$\psi\in\fa_\theta^*$ be such that
\[
    \psi>0
    \quad\text{on } \L_f-\{0\}.
\]
Define
\[
    \d_\psi(x,y)
    :=
    \psi(\mu(f(x)^{-1}f(y)))
    \quad\text{for } x,y\in Y.
\]
It was shown in \cite{DKO_AR} that $\d_\psi$ behaves like a metric on
$Y$; for example, it satisfies a coarse triangle inequality.  The following
is a consequence of Lemma \ref{lem.DKO_additive} and the Gromov
hyperbolicity of $Y$.  Since $\psi>0$ on $\L_f-\{0\}$, there
exists $c>0$ such that
\begin{equation} \label{eqn.lbdpsi}
    \d_\psi(x,y)>-c
    \quad\text{for all } x,y\in Y.
\end{equation}

\begin{proposition}
There exists $C>0$ such that, for all $x,y,z\in Y$,
\[
    \d_\psi(x,z)
    \le
    \d_\psi(x,y)+\d_\psi(y,z)+C.
\]
\end{proposition}

\begin{proof}
Let $w\in\pi_{[x,z]}(y)$ be a nearest-point projection of $y$ to a
geodesic $[x,z]$.  By Lemma \ref{lem.DKO_additive},
\[
    \d_\psi(x,z)
    \le
    \d_\psi(x,w)+\d_\psi(w,z)+D_0\|\psi\|,
\]
where $D_0$ is the constant from Lemma \ref{lem.DKO_additive}.  Since the
concatenation $[x,w]\cup[w,y]$ is a uniform quasi-geodesic, Lemmas
\ref{lem.stableqg} and \ref{lem.DKO_additive} imply that
\[
    \d_\psi(x,y)
    \approx
    \d_\psi(x,w)+\d_\psi(w,y)
    >
    \d_\psi(x,w)-c,
\]
where $c$ is as in \eqref{eqn.lbdpsi}.  Similarly,
\[
    \d_\psi(y,z)
    \approx
    \d_\psi(y,w)+\d_\psi(w,z)
    >
    \d_\psi(w,z)-c.
\]
Combining these three estimates gives the desired inequality.
\end{proof}

We also note that the metric-like function $\d_{\psi}$ behaves quasi-isometrically to the Gromov model $(Y, d)$ as follows: 

\begin{lemma}  \label{lem.QIpropsofmetriclike}
\begin{enumerate}
    \item For any $x, y, z \in Y$, if $y \in [x, z]$, then
    $$
    \d_{\psi}(x, z) \approx \d_{\psi}(x, y) + \d_{\psi}(y, z).
    $$
    \item There exist $a,b>0$ such that 
for all $x,y\in Y$,
    $$
    a \, d(x,y)-b
    \le \d_\psi(x,y)
    \le b \, d(x,y)+b.$$
    \item 
There exists $A>0$ such that, for every $\xi\in\partial Y$ and every
$s\ge0$, there exists $x \in [o_Y, \xi]$ such that 
\[
    |\d_\psi(o_Y,x)-s|\le A.
\]
\end{enumerate}
\end{lemma}
\begin{proof}
(1) is an immediate consequence of Lemma \ref{lem.DKO_additive}.

Since $\psi>0$ on
$\L_f-\{0\}$ and $f$ is a quasi-isometric embedding, there exist $a,b>0$ such that
\begin{equation}\label{eqn.dpsi-linear}
    a\,d(x,y)-b
    \le \d_\psi(x,y)
    \le b\,d(x,y)+b
\end{equation}
for all $x,y\in Y$.   This shows (2).

Finally, fix $\xi\in\partial Y$ and denote by $\xi_t \in [o_Y, \xi]$ the point such that $d(o_Y, \xi_t) = t$ for $t \ge 0$. We then set
\[
    F(t):=\d_\psi(o_Y,\xi_t).
\]
Note that definition of $F(t)$ involves a choice of $[o_Y, \xi]$, but different choices only make uniformly bounded error which is allowed for our purpose.
By Lemma \ref{lem.DKO_additive}, if $0\le t\le u$, then
\[
    F(u) \approx F(t)+\d_\psi(\xi_t,\xi_u).
\]
Together with \eqref{eqn.dpsi-linear}, this gives constants $a_1,b_1>0$,
independent of $\xi$, such that for all $0\le t\le u$,
\begin{equation}\label{eqn.F-coarse-linear}
    a_1(u-t)-b_1
    \le F(u)-F(t)
    \le b_1(u-t)+b_1.
\end{equation}
Since $F(0)=0$ and $F(t)\to\infty$ as $t \to \infty$, for any $s\ge0$ choose the smallest
integer $n\ge0$ with $F(n)\ge s$.  If $n=0$, then $s=0$.  If $n\ge1$, then
$F(n-1)<s$, and hence by \eqref{eqn.F-coarse-linear},
\[
    0\le F(n)-s\le F(n)-F(n-1)\le 2b_1 .
\]
Thus (3) holds with $A=2b_1$.
\end{proof}

For subsets $E,F\subset Y$, set
\[
    \d_\psi(E,F)
    :=
    \inf_{x\in E,\,y\in F}\d_\psi(x,y).
\]

\begin{lemma} \label{lem.closestinhoroball}
Let $\mathsf P\in\cal P$ and let $p=\xi_{\mathsf P}$.  For every
$x\in[o_Y,p]$, we have
\[
    \d_\psi(x,\Ga o_Y)
    \approx
    \d_\psi(x,o_Y),
\]
where the implied constant depends only on the ambient constants and
$\psi$.
\end{lemma}

\begin{proof} Let $B_p\in\cal B$ be the horoball based at $p$.  We first claim that
\[
    \d_\psi(x,\Ga o_Y)
\approx
    \d_\psi(x,\mathsf P o_Y).
\]
It suffices to consider the case that $x$ is sufficiently deep in $B_p$.

Since $\mathsf P$ acts cocompactly on $\partial B_p$, there exists
$A>0$ such that every point of $\partial B_p$ is within distance $A$ of
some point of $\mathsf P o_Y$.  Let $h\in\Ga$.  A geodesic from $x$ to
$h o_Y$ leaves $B_p$ through a point $u_h\in\partial B_p$.  Choose
$\gamma_h\in\mathsf P$ such that
\[
    d(u_h,\gamma_h o_Y)\le A.
\]
By Lemma \ref{lem.cptcartan},
\[
    \d_\psi(x,\gamma_h o_Y)
    \approx
    \d_\psi(x,u_h).
\]
On the other hand, since $u_h\in[x,h o_Y]$, Lemma
\ref{lem.DKO_additive} and \eqref{eqn.lbdpsi} give
\[
    \d_\psi(x,h o_Y)
    \approx
    \d_\psi(x,u_h)+\d_\psi(u_h,h o_Y)
    \ge
    \d_\psi(x,u_h)-c
\]
where $c > 0$ is given in \eqref{eqn.lbdpsi}.
Hence
\[
    \d_\psi(x,\gamma_h o_Y)
    \le
    \d_\psi(x,h o_Y) + c' 
\]
for some uniform constant $c' > 0$ independent of $h$ and $x$. Since this holds for all $h \in \Ga$ and $\mathsf{P} < \Ga$, the claim follows.

By the claim, we can choose $\ga \in \mathsf{P}$ such that $\d_{\psi}(x, \Ga o_Y) \approx \d_{\psi}(x, \ga o_Y)$.
Let $w\in\pi_{[o_Y,p]}(\ga o_Y)$.  By tautness of $Y$ and Lemma
\ref{lem.projinhoroball}, we may apply Lemma
\ref{lem.parabolictransmiddle} to obtain a point
$y\in[o_Y,\ga^{-1}o_Y]$ such that the points $w$, $\ga^{-1}w$, and $y$
are uniformly close.  Since $f$ is a quasi-isometric embedding, the points
$f(w)$, $f(\ga^{-1}w)$, and $f(y)$ are also uniformly close.  Hence
\[
    \mu_\theta(f(w)^{-1})
    \approx
    \mu_\theta(f(\ga^{-1}w)^{-1})
    =
    \mu_\theta(f(w)^{-1}f(\ga o_Y)).
\]

We distinguish two cases according to the relative positions of $w$ and
$x$ on $[o_Y,p]$. First suppose that $w\in[o_Y,x]$.  By Lemma \ref{lem.DKO_additive},
\[
\begin{aligned}
    \mu_\theta(f(x)^{-1})
    &\approx
    \mu_\theta(f(x)^{-1}f(w))
    +
    \mu_\theta(f(w)^{-1}) \\
    &\approx
    \mu_\theta(f(x)^{-1}f(w))
    +
    \mu_\theta(f(w)^{-1}f(\ga o_Y)) \\
    &\approx
    \mu_\theta(f(x)^{-1}f(\ga o_Y)),
\end{aligned}
\]
since the concatenation $[x,w]\cup[w,\ga o_Y]$ is a
$(1,O(\delta))$-quasi-geodesic.  Applying $\psi$, we get
\[
    \d_\psi(x,o_Y)
    \approx
    \d_\psi(x,\ga o_Y)
    \approx
    \d_\psi(x,\Ga o_Y),
\]
as desired.

Now suppose that $x\in[o_Y,w]$.  Again by Lemma \ref{lem.DKO_additive},
\[
\begin{aligned}
    \mu_\theta(f(x)^{-1}f(\ga o_Y))
    &\approx
    \mu_\theta(f(x)^{-1}f(w))
    +
    \mu_\theta(f(w)^{-1}f(\ga o_Y)) \\
    &\approx
    \mu_\theta(f(x)^{-1}f(w))
    +
    \mu_\theta(f(w)^{-1}) \\
    &\approx
    \mu_\theta(f(x)^{-1}f(w))
    +
    \mu_\theta(f(w)^{-1}f(x))
    +
    \mu_\theta(f(x)^{-1}).
\end{aligned}
\]
Applying $\psi$ gives
\[
\begin{aligned}
    \d_\psi(x,\ga o_Y)
    &\approx
    \d_\psi(x,w)+\d_\psi(w,x)+\d_\psi(x,o_Y) \\
    &>
    \d_\psi(x,o_Y)-2c,
\end{aligned}
\]
where the last inequality follows from \eqref{eqn.lbdpsi}.  Since
$\d_\psi(x,\ga o_Y)\approx\d_\psi(x,\Ga o_Y)$ and
$o_Y\in\Ga o_Y$, the desired estimate follows in this case as well.
\end{proof}

\subsection*{$\fa_\theta$-valued Gromov products in shadows}

We now translate the elementary comparison between shadows and Gromov
products in the Gromov hyperbolic space $Y$ into a comparison involving
$\fa_\theta$-valued Gromov products on the flag variety.  The main input
from the relatively Morse property is that nearest-point projections in
$Y$ coarsely control the higher-rank Gromov product
$\cal G^\theta(f(\xi),f(\eta))$ through Cartan projections; see Lemma
\ref{lem.gromovproductnearest}.  Together with the coarse additivity of
Cartan projections along geodesics in $Y$, Lemma \ref{lem.DKO_additive}, this
gives the following analogue of Lemma \ref{lem.shadowandgromprod}.  For
Anosov groups, this comparison was proved in \cite{DKO_AR}; the same
argument applies in the relatively Morse setting.

For $\psi\in\fa_\theta^*$, write
\[
    \bar\psi:=\frac{\psi+\psi\circ\i}{2}.
\]

\begin{prop} \label{prop.shadowcompare}
Let $\psi\in\fa_\theta^*$ be such that $ \psi>0
$ on $\L_f-\{0\}$.
\begin{enumerate}
    \item For every $R>0$, there exists $c_1>0$ such that, for every
    $\xi\in\partial Y$ and every $z\in[o_Y,\xi]$, if
    $\eta\in O_R(o_Y,z)-\{\xi\}$, then
    \[
        \psi(\cal G^\theta(f(\xi),f(\eta)))
        \ge
        \d_{\bar\psi}(o_Y,z)-c_1.
    \]

    \item For every $c_2>0$, there exists $R>0$ such that, for every
    $\xi\in\partial Y$ and every $z\in[o_Y,\xi]$, if
    $\eta\in\partial Y-\{\xi\}$ satisfies
    \[
        \psi(\cal G^\theta(f(\xi),f(\eta)))
        \ge
        \d_{\bar\psi}(o_Y,z)+c_2,
    \]
    then
    \[
        \eta\in O_R(o_Y,z).
    \]
\end{enumerate}
\end{prop}

\begin{proof}
We first prove (1).  Fix $R>0$, $\xi\in\partial Y$, and
$z\in[o_Y,\xi]$.  Let $\eta\in O_R(o_Y,z) - \{\xi\}$ and
$y\in\pi_{[\xi,\eta]}(o_Y)$.  Let $[z,\xi]\subset[o_Y,\xi]$.

We claim that there exists $r=r(R)>0$ such that
\begin{equation} \label{c1}
    [z,\xi]\cap B(y,r)\ne\emptyset.
\end{equation}
Since $\eta\in O_R(o_Y,z)$, Lemma \ref{lem.stableqg} implies that the
uniform quasi-geodesic
\[
    [o_Y,y]\cup[y,\eta]
\]
intersects $B(z,R+O(\delta))$.  Choose
$w\in[o_Y,y]\cup[y,\eta]$ such that
\[
    d(z,w)<R+O(\delta).
\]
We consider two cases, according to the position of $w$ relative to $y$.

If $w\in[o_Y,y]$, then both
\[
    [w,y]\cup[y,\xi]
    \quad\text{and}\quad
    [w,z]\cup[z,\xi]
\]
are uniform quasi-geodesics with the same endpoints.  By
Lemma \ref{lem.stableqg}, the distance from $y$ to
$[w,z]\cup[z,\xi]$ is bounded above by a uniform constant depending on $R$.  Since $[w,z]$ has length at most $R+O(\delta)$, the
claim follows in this case.

If $w\in[y,\eta]$, choose a geodesic $[w,\xi]\subset[\xi,\eta]$.  Then
$y\in[w,\xi]$.  As in the previous case,
\[
    [z,w]\cup[w,\xi]
\]
is a  uniform quasi-geodesic.  Hence, by Lemma
\ref{lem.stableqg}, the point $y\in[w,\xi]$ lies in a uniform
neighborhood of $[z,\xi]$. Now the claim follows.

By \eqref{c1}, there exists $u\in[z,\xi]\subset[o_Y,\xi]$ such that
$d(y,u)<r$.  Since $f$ is a quasi-isometric embedding, Lemma
\ref{lem.cptcartan} implies that
\[
    \mu(f(y))\approx\mu(f(u)),
\]
where the implied constant depends on $r$.  By Lemma
\ref{lem.gromovproductnearest},
\[
    \psi(\cal G^\theta(f(\xi),f(\eta)))
    \approx
    \d_{\bar\psi}(o_Y,y)
    \approx
    \d_{\bar\psi}(o_Y,u).
\]
Since $u\in[z,\xi]\subset[o_Y,\xi]$, Lemma \ref{lem.DKO_additive} gives
\[
    \d_{\bar\psi}(o_Y,u)
    \approx
    \d_{\bar\psi}(o_Y,z)+\d_{\bar\psi}(z,u).
\]
By \eqref{eqn.lbdpsi}, the term $\d_{\bar\psi}(z,u)$ is bounded from
below by a uniform constant.  Combining these estimates, we obtain a
constant $c_1>0$ such that
\[
    \psi(\cal G^\theta(f(\xi),f(\eta)))
    \ge
    \d_{\bar\psi}(o_Y,z)-c_1.
\]
This proves (1).

We now prove (2).  Let $c_2>0$, let $\xi\in\partial Y$, and let
$z\in[o_Y,\xi]$.  Suppose that $\eta\in\partial Y - \{\xi\}$ satisfies
\[
    \psi(\cal G^\theta(f(\xi),f(\eta)))
    \ge
    \d_{\bar\psi}(o_Y,z)+c_2.
\]
We claim that there exists $r_0=r_0(c_2)>0$ such that, for every
$y\in\pi_{[\xi,\eta]}(o_Y)$, the union
\[
    [o_Y,y]\cup[y,\eta]
\]
intersects $B(z,r_0)$.  Since $[o_Y,y]\cup[y,\eta]$ and
$[o_Y,\eta]$ have uniformly bounded Hausdorff distance by Lemma
\ref{lem.stableqg}, the claim implies that
\[
    \eta\in O_R(o_Y,z)\quad\text{ for $R=r_0+O(\delta)$.}
\]

Let $C$ be the constant from Lemma \ref{lem.gromovproductnearest}.  Then
\begin{equation} \label{eqn.proofshadowballcompare}
    \d_{\bar\psi}(o_Y,y)
    \ge
    \d_{\bar\psi}(o_Y,z)+c_2-\|\psi\|C.
\end{equation}
By Lemma \ref{lem.stableqg}, there exists $w\in[o_Y,\xi]$ such that
$d(y,w)<O(\delta)$.  Since $f$ is a quasi-isometric embedding,
\begin{equation} \label{eqn.proofshadowballcompare2}
    \d_{\bar\psi}(o_Y,w)
    \approx
    \d_{\bar\psi}(o_Y,y).
\end{equation}

We distinguish two cases according to the position of $w$ relative to $z$.
If $w\in[z,\xi]\subset[o_Y,\xi]$, then $z\in[o_Y,w]$.  Since
$d(y,w)<O(\delta)$, the concatenation
\[
    [o_Y,w]\cup[w,y]
\]
is a $(1,O(\delta))$-quasi-geodesic.  Hence, by Lemma
\ref{lem.stableqg}, the geodesic $[o_Y,y]$ intersects a uniform
neighborhood of $z$.  Thus $[o_Y,y]\cup[y,\eta]$ intersects a uniform
neighborhood of $z$.

If $w\in[o_Y,z]\subset[o_Y,\xi]$, then Lemma \ref{lem.DKO_additive}
gives
\[
    \d_{\bar\psi}(o_Y,z)
    \approx
    \d_{\bar\psi}(o_Y,w)+\d_{\bar\psi}(w,z).
\]
Together with \eqref{eqn.proofshadowballcompare} and
\eqref{eqn.proofshadowballcompare2}, this implies that
$\d_{\bar\psi}(w,z)$ is uniformly bounded above.  Since $\psi>0$ on
$\L_f-\{0\}$, the same is true for $\bar\psi$, and therefore
$\|\mu(f(w)^{-1}f(z))\|$ is uniformly bounded.  Since $f$ is a
quasi-isometric embedding, $d(w,z)$ is uniformly bounded.  As
$d(y,w)<O(\delta)$, it follows that $[o_Y,y]\cup[y,\eta]$ intersects a
uniform neighborhood of $z$.

In both cases the claim follows, and this completes the proof.
\end{proof}

\section{Counting in parabolic subgroups}

\label{sec.paraboliccounting}

The counting estimates in this section require only the relatively
Anosov hypothesis; no Morse embedding is assumed.  The linear form
$\psi$ need not be invariant under opposition.  They extend the mixed
exponential growth estimates for parabolic subgroups in Hilbert
geometry \cite[Proposition 7.7]{BT_global} to Cartan lengths defined
by arbitrary positive linear forms in semisimple Lie groups.

Throughout this section, let $\Ga<G$ be $\theta$-Anosov relative to a
finite collection $\cal P$.  Let $\L_\Ga\subset\fa^+$ denote the limit
cone of $\Ga$, that is, the asymptotic cone of $\mu(\Ga)$.

Fix $\psi\in\fa_\theta^*$ that is positive on $\L_\Ga-\{0\}$.
For a subgroup $H<\Ga$, let $\delta_\psi(H)$ denote the abscissa of
convergence of the Poincar\'e series
\[
    s\mapsto \sum_{\ga\in H} e^{-s\psi(\mu(\ga))}.
\]
We prove the following counting estimate for parabolic subgroups.

\begin{theorem} \label{thm.countinginsegments}
Let $\mathsf P\in\cal P$.  Then there exist
$a_\psi(\mathsf P)\in\Z_{\ge 0}$, $C>1$, $k\in\N$, and $T_0>0$ such that,
for all $T>T_0$ and all $n\ge 0$,
\[
\begin{aligned}
& C^{-1} e^{\delta_\psi(\mathsf P)(T+kn)}
    (1+T+kn)^{a_\psi(\mathsf P)}
\\
&\le
\#\left\{
    \ga\in\mathsf P:
    T+kn\le \psi(\mu(\ga))<T+k(n+1)
\right\}
\\
&\le
C e^{\delta_\psi(\mathsf P)(T+kn)}
    (1+T+kn)^{a_\psi(\mathsf P)} .
\end{aligned}
\]
Moreover, if $\mathsf P$ is virtually cyclic or if $G$ has rank one, then
$a_\psi(\mathsf P)=0$.
\end{theorem}

We deduce Theorem \ref{thm.countinginsegments} from the following 
counting estimate. We write $\ll$ for inequality up to a uniform multiplicative constant, and similarly for $\gg$. We write $\asymp$ when we have both $\ll$ and $\gg$.

\begin{proposition} \label{prop.countingperipheral}
Let $\mathsf P\in\cal P$.  Then there exist
$a_\psi(\mathsf P)\in\Z_{\ge 0}$ and $c>0$ such that, for all sufficiently
large $T$,
\[
    e^{\delta_\psi(\mathsf P)T}(T-c)^{a_\psi(\mathsf P)}
    \ll
    \#\{\ga\in\mathsf P:\psi(\mu(\ga))\le T\}
    \ll
    e^{\delta_\psi(\mathsf P)T}(T+c)^{a_\psi(\mathsf P)} .
\]
Moreover, if $\mathsf P$ is virtually cyclic  or if $G$ has rank one, then
$a_\psi(\mathsf P)=0$.
\end{proposition}

The key input is a volume estimate of Benoist-Oh \cite{BO}.  We first
recall the structural description of parabolic subgroups in the relatively
Anosov setting.  After the reduction in \cite[Section 4.4]{CZZ_relative},
we may assume that $P_\theta$ contains no simple factor of $G$.  Let
$\mathsf P\in\cal P$.  By \cite[Theorem 4.4]{CZZ_relative}, there exists a
closed subgroup $H<G$ with finitely many connected components such that
$\mathsf P$ is a cocompact lattice in $H$.  Moreover, if $U$ denotes the
unipotent radical of $H$, then
\be \label{eqn.structure}
    H=L\ltimes U \quad \text{and} \quad  H^\circ=L^\circ\times U,
    \ee
where $L<H$ is compact and its identity component $L^\circ$ is abelian.

Let $\mathfrak u=\Lie(U)$, and let $\ell$ denote Lebesgue measure on
$\mathfrak u$.  We shall use the following consequence of \cite{BO}:

\begin{theorem} \label{thm.volumegrowth}
Let $R_1,\ldots,R_m$ be positive rational functions on $\mathfrak u$ which
are defined everywhere.  Let $c_1,\ldots,c_m\in\R$, and suppose that
$R:=R_1^{c_1}\cdots R_m^{c_m}$
is a proper function.  Then
\[
    \ell\bigl(\{Y\in\mathfrak u:R(Y)\le T\}\bigr)
    \sim c_0 T^r(\log T)^q
\]
for some $c_0>0$, $r\ge 0$, and $q\in\Z_{\ge 0}$.
\end{theorem}
\begin{proof}
This is immediate from \cite[Proposition 7.2]{BO} when all exponents
\(c_i\) are positive rational numbers. The same proof also gives the
present form; we briefly indicate the additional points needed for real
exponents.
Following the notation in \cite[Proposition 7.2]{BO}, let \(Z\) be the
affine space over \(\mathbb R\) associated to the real vector space
\(\mathfrak u\), so that \(Z(\mathbb R)=\mathfrak u\). Let \(\omega\) be the
standard algebraic volume form on \(Z\) inducing \(\ell\).  Write \(R_i=f_i/g_i\), with \(f_i\) and \(g_i\)
polynomial. Choose a smooth projective real compactification \(V\) of \(Z\).
After replacing \(V\) by a resolution, and keeping the same notation, we may
assume that \(D:=V-Z\), the divisors of the rational functions \(f_i,g_i\),
and the divisor of \(\omega\), have simple normal crossings. Put \(F=R\)
and \(f=1/F\).

Let \(y_0\in D(\mathbb R)\cap\overline{Z(\mathbb R)}\). In a neighborhood
of \(y_0\), choose local real analytic coordinates \((x_1,\ldots,x_N)\) such
that
\[
    D=\{x_1\cdots x_r=0\}.
\]
By the above resolution of singularities, and arguing as in the proof of
\cite[Lemma 6.7]{CZZ_relative}, we may write locally

\[
    f_i=x_1^{a_{i1}}\cdots x_r^{a_{ir}}\widehat f_i,
    \quad \text{and} \quad
    g_i=x_1^{b_{i1}}\cdots x_r^{b_{ir}}\widehat g_i,
\]
where \(a_{ij},b_{ij}\in\mathbb Z\) and
$\widehat f_i,\widehat g_i$ are nowhere-vanishing real analytic functions.
After restricting to one orthant and replacing \(x_j\) by \(|x_j|\), we may
assume \(x_j>0\) for \(j=1,\ldots,r\). Then
$R_i=\widehat R_i\prod_{j=1}^r x_j^{a_{ij}-b_{ij}}$ where $\widehat R_i := \widehat f_i / \widehat g_i$,
and hence
\[
    f=R^{-1}=\widehat f\prod_{j=1}^r x_j^{s_j}.
\]
for some \(s_j\in\mathbb R\) and some positive real analytic unit
\(\widehat f\). 
 Since \(R\) is proper on \(Z(\mathbb R)\), \(f=R^{-1}\)
tends to \(0\) along every boundary component meeting the closure of
\(Z(\mathbb R)\). Hence \(s_j>0\) for every relevant boundary component.

Moreover, locally
\[
    \omega
    =
    a_0(x)\prod_{j=1}^r x_j^{\beta_j}\,dx_1\cdots dx_N
\]
with \(\beta_j\in\mathbb Z\) and \(a_0\) a nowhere-vanishing real analytic
function.

Thus the local integrals appearing in the proof of
\cite[Proposition 7.2]{BO} are of monomial type.
 The remainder of their argument applies without change.
\end{proof}

\subsection*{Proof of Proposition \ref{prop.countingperipheral}}

By \cite[Proposition 2.3]{CZZ_relative}, for each $\alpha\in\theta$ there
exist $m_\alpha\in\N$, $C_\alpha>1$, and an everywhere-defined positive
rational function $R_\alpha:\mathfrak u\to\R$ such that, for all
$Y\in\mathfrak u$,
\[
    C_\alpha^{-1}R_\alpha(Y)^{1/m_\alpha}
    \le
    e^{\omega_\alpha(\mu(\exp Y))}
    \le
    C_\alpha R_\alpha(Y)^{1/m_\alpha},
\]
where $\omega_\alpha\in\fa_\theta^*$ is the fundamental weight associated
to $\alpha$.

Write
\[
    \psi=\sum_{\alpha\in\theta} c_\alpha\omega_\alpha
\]
for coefficients $c_\alpha\in\R$, $\alpha \in \theta$.
Set
\[
    R_\psi:=\prod_{\alpha\in\theta} R_\alpha^{c_\alpha/m_\alpha}
    \quad \text{and} \quad
    C_\psi:=\prod_{\alpha\in\theta} C_\alpha^{|c_\alpha|}.
\]
Then, for all $Y\in\mathfrak u$,
\[
    C_\psi^{-1}R_\psi(Y)^{-1}
    \le
    e^{-\psi(\mu(\exp Y))}
    \le
    C_\psi R_\psi(Y)^{-1}.
\]
By \cite[Lemma 7.3]{CZZ_relative}, the function $R_\psi$ is proper.

Let
\[
    \mathsf P_1 := \pi(\mathsf P\cap H^\circ),
\]
where $\pi:H^\circ\to U$ is the projection.  Then $\mathsf P_1$ is a
cocompact lattice in $U$.  Since $\mathsf P\cap H^\circ$ has finite index
in $\mathsf P$ and $\ker\pi$ is compact, Lemma~\ref{lem.cptcartan}
gives a constant $c>0$ such that
\[
\begin{aligned}
    \#\{g\in\mathsf P_1:\psi(\mu(g))\le T-c\}
    &\ll \#\{\ga\in\mathsf P:\psi(\mu(\ga))\le T\}\\
    &\ll \#\{g\in\mathsf P_1:\psi(\mu(g))\le T+c\}.
\end{aligned}
\]
It therefore suffices to prove the stated growth estimate for
$\mathsf P_1$, since bounded shifts of $T$ do not change that estimate.

Let $\la_U$ denote a Haar measure on $U$.
Let $Q_1\subset U$ be a bounded open set such that
the translates $gQ_1$, $g\in\mathsf P_1$, are pairwise disjoint.  Then
\[
\begin{aligned}
\#\{g\in\mathsf P_1:\psi(\mu(g))\le T\}
&\ll
\lambda_U\left(
    \bigcup_{\substack{g\in\mathsf P_1\\ \psi(\mu(g))\le T}} gQ_1
\right) \\
&\ll
\lambda_U\left(
    \{u\in U:\psi(\mu(u))\le T+c'\}
\right) \\
&\ll
\ell\left(
    \{Y\in\mathfrak u:R_\psi(Y)\le c e^T\}
\right)
\end{aligned}
\]
for some constants $c',c>1$.

Similarly, choose a compact set $Q_2\subset U$ such that
$\mathsf P_1Q_2=U$.  Then, after increasing $c',c$ with $c>1$, if necessary,
\[
\begin{aligned}
\#\{g\in\mathsf P_1:\psi(\mu(g))\le T\}
&\gg
\lambda_U\left(
    \bigcup_{\substack{g\in\mathsf P_1\\ \psi(\mu(g))\le T}} gQ_2
\right) \\
&\gg
\lambda_U\left(
    \{u\in U:\psi(\mu(u))\le T-c'\}
\right) \\
&\gg
\ell\left(
    \{Y\in\mathfrak u:R_\psi(Y)\le c^{-1}e^T\}
\right).
\end{aligned}
\]
The proposition now follows from Theorem \ref{thm.volumegrowth}, together
with the fact that the exponential growth rate of
\[
    \#\{g\in\mathsf P_1:\psi(\mu(g))\le T\}
\]
is $\delta_\psi(\mathsf P)$.  If $\mathsf P$ is virtually cyclic, then the
corresponding unipotent group is one-dimensional, and the volume asymptotic
has no logarithmic factor.  Hence $a_\psi(\mathsf P)=0$. See Theorem \ref{thm.rankone_no_log} for the rank one case.
\qed

\medskip

Before proving Theorem \ref{thm.countinginsegments}, we record the
following entropy gap for parabolic subgroups, due to
Canary--Zhang--Zimmer.

\begin{theorem} \cite[Lemma 7.4, Theorem 6.1, Theorem 7.1]{CZZ_relative}
\label{thm.entropydrop}
For each $\mathsf P\in\cal P$, we have
\[
    0<\delta_\psi(\mathsf P)<\delta_\psi(\Ga).
\]
\end{theorem}

\subsection*{Proof of Theorem \ref{thm.countinginsegments}}

By Proposition \ref{prop.countingperipheral}, there exist $c>0$ and
$T_0>0$ such that, for all $T>T_0$,
\[
    e^{\delta_\psi(\mathsf P)T-c}(1+T)^{a_\psi(\mathsf P)}
    \le
    \#\{\ga\in\mathsf P:\psi(\mu(\ga))<T\}
    \le
    e^{\delta_\psi(\mathsf P)T+c}(1+T)^{a_\psi(\mathsf P)}.
\]
For simplicity, write
\[
    \delta=\delta_\psi(\mathsf P)
    \quad \text{and} \quad
    a=a_\psi(\mathsf P),
\]
and define
\[
    N(T):=\#\{\ga\in\mathsf P:\psi(\mu(\ga))<T\}.
\]

Choose $k\in\N$ so large that
\[
    e^{\delta k-c}-e^c>0,
\]
which is possible since $\delta >0$ by Theorem
\ref{thm.entropydrop}.  Set $S:=T+kn$.  Then
\[
\begin{aligned}
N(S+k)-N(S)
&\le
e^{\delta(S+k)+c}(1+S+k)^a
-
e^{\delta S-c}(1+S)^a  \\
&=
e^{\delta S}(1+S)^a
\left[
    e^{\delta k+c}
    \left(1+\frac{k}{1+S}\right)^a
    -
    e^{-c}
\right] \\
&\le
C e^{\delta S}(1+S)^a
\end{aligned}
\]
for a constant $C>1$ independent of $T$ and $n$.  Similarly,
\[
\begin{aligned}
N(S+k)-N(S)
&\ge
e^{\delta(S+k)-c}(1+S+k)^a
-
e^{\delta S+c}(1+S)^a  \\
&=
e^{\delta S}(1+S)^a
\left[
    e^{\delta k-c}
    \left(1+\frac{k}{1+S}\right)^a
    -
    e^c
\right] \\
&\ge
C^{-1} e^{\delta S}(1+S)^a,
\end{aligned}
\]
after increasing $C$, if necessary.  Since
\[
    N(S+k)-N(S)
    =
    \#\{\ga\in\mathsf P:S\le \psi(\mu(\ga))<S+k\},
\]
and $S=T+kn$, this proves the theorem.
\qed

\subsection*{Rank-one parabolic subgroups}

We record a rank-one refinement of Proposition
\ref{prop.countingperipheral}, which shows the absence of the polynomial term in Theorem \ref{thm.countinginsegments} for rank-one $G$. In this subsection, 
assume that $G$ has real rank one.  Let $\Pi=\{\alpha\}$, so that
$\theta=\{\alpha\}$ and $\fa_\theta=\fa \simeq \R$.  
Let $\mathfrak{n} := \Lie(N)$ and write 
\[
    \mathfrak n=\mathfrak g_\alpha\oplus\mathfrak g_{2\alpha},
\]
where $\mathfrak{g}_{\alpha}$ and $\mathfrak{g}_{2\alpha}$ are the corresponding root spaces, 
with the convention that $\mathfrak g_{2\alpha}=0$ if $2\alpha$ is not a
root.

Let
$\mathsf P\in\cal P$. Let $H=L\ltimes U$ be the subgroup associated to $\mathsf P$ as in \eqref{eqn.structure}. Up to conjugation, we may assume that $U<N$, and use the same notation  $\mathfrak u:=\Lie(U)$ as before.  Define
\[
    V_{\mathsf P}:=\operatorname{pr}_{\mathfrak g_\alpha}(\mathfrak u),
    \quad
    Z_{\mathsf P}:=\mathfrak u\cap\mathfrak g_{2\alpha},\quad \text{and} \quad
    Q(\mathsf P):=\dim V_{\mathsf P}+2\dim Z_{\mathsf P}.
\]

Since $\fa$ is one-dimensional, $\L_{\Ga} = \fa^+$ and hence any $\psi\in \fa^*$ positive on $\L_\Ga-\{0\}$ is multiplication by a positive real number.
The norm $\|\psi\|$ is given by $\psi(H_0)$ where $H_0\in \fa^+$ is the unique unit vector.
\begin{theorem} 
\label{thm.rankone_no_log} Suppose that $\text{rank } G=1$. For any positive $\psi\in\fa^*$ on $\fa^+-\{0\}$, we have 
\[
    \#\{\ga\in\mathsf P:\psi(\mu(\ga))\le T\}
    \asymp
    e^{\delta_\psi(\mathsf P)T},
\]
and 
   $\delta_\psi(\mathsf P) = \frac{\norm{\alpha} Q(\mathsf{P})}{2 \norm{\psi}}$. In particular, in Theorem \ref{thm.countinginsegments}, we have
\[
    a_\psi(\mathsf P)=0  = a_{\bar\psi}(\mathsf P).
\]
\end{theorem}

\begin{proof} Let $H_0 \in \fa^+$ be the unit vector. Let $b:=\alpha(H_0)$ and $\la_\psi:=\psi(H_0)$.
We use the same reduction as in Proposition
\ref{prop.countingperipheral}. Recall that 
$ \pi:H^\circ\to U$ is the projection and 
$\mathsf P_1 =\pi(\mathsf P\cap H^\circ)$.
Then $\mathsf P_1$ is a cocompact lattice in $U$.  Since
$\mathsf P\cap H^\circ$ has finite index in $\mathsf P$ and $\ker\pi$ is
compact, Lemma \ref{lem.cptcartan} implies that counting $\mathsf P$ and
counting $\mathsf P_1$ give the same estimates, up to multiplicative constants, as before.

Write $u=\exp(X+Z)\in U$ with
$X\in\mathfrak g_\alpha$ and $Z\in\mathfrak g_{2\alpha}$.  By \cite[(2.5)]{Papageorgiou2025}, following
\cite[p.~72]{Rouviere2003},  
\[
\cosh ^2\left(\frac {d_{NA} (o, uo)}2\right)
=
\left(1+\frac{\norm{X}^2}{8}\right)^2
+
\frac14\norm{Z}^2
\] where $d_{NA}$ is the distance induced from the norm whose unit vector $H_1$ satisfies $\alpha(H_1)=1/2$.
It follows that $d_{NA}$ is the $2b$ multiple of the Riemannian distance $d$ induced from the norm
for which $H_0$ is a unit vector.

Therefore
\[
d(o,uo)\approx
\frac{2}{b}
\log \max\{1,\norm{X},\norm{Z}^{1/2}\}.
\]

 Define
\[
    \norm{u}_{\mathrm{cusp}}
    :=
    \max\{1,\norm{X},\norm{Z}^{1/2}\},
\]
where the term involving $Z$ is omitted if $\mathfrak g_{2\alpha}=0$.
Since $\mu(u)=d(o,uo)H_0$, we get
\[
    \psi(\mu(u))
    =
    \lambda_\psi d(o,uo) \approx
    \frac{2\lambda_\psi}{b}\log\norm{u}_{\mathrm{cusp}}. 
\]
Hence
\[
    \psi(\mu(u))\le T
    \quad\Longleftrightarrow\quad
    \norm{u}_{\mathrm{cusp}}
    \ll
    e^{bT/(2\lambda_\psi)},
\]
up to changing the implicit constants.

It remains to compute the volume growth of these $\norm{\cdot}_{\mathrm{cusp}}$-balls inside $U$.  Choose a linear complement $W$ to
$Z_{\mathsf P}$ in $\mathfrak u$.  The projection
$W\to V_{\mathsf P}$ is an isomorphism, so every $Y\in\mathfrak u$ can be
written uniquely as
\[
    Y=v+\phi(v)+z \quad \text{with }
    v\in V_{\mathsf P} \text{ and } z\in Z_{\mathsf P},
\]
for some linear map $\phi:V_{\mathsf P}\to\mathfrak g_{2\alpha}$.  The
condition
\[
    \norm{\exp Y}_{\mathrm{cusp}}\le R
\]
is equivalent, up to uniform constants, to
\[
    \norm{v}\ll R \quad \text{and} \quad
    \norm{\phi(v)+z}\ll R^2.
\]
Thus the $v$-variables contribute $R^{\dim V_{\mathsf P}}$ and the
$z$-variables contribute $R^{2\dim Z_{\mathsf P}}$.  Therefore
\[
    \lambda_U (\{u\in U:\norm{u}_{\mathrm{cusp}}\le R\})
    \asymp
    R^{Q(\mathsf P)}.
\]

Since $\mathsf P_1$ is a cocompact lattice in $U$, a compact fundamental
domain comparison gives
\[
    \#\{g\in\mathsf P_1:\norm{g}_{\mathrm{cusp}}\le R\}
    \asymp
    R^{Q(\mathsf P)}.
\]
Substituting $R=e^{bT/(2\lambda_\psi)}$ yields
\[
    \#\{g\in\mathsf P_1:\psi(\mu(g))\le T\}
    \asymp
    e^{\frac{bQ(\mathsf P)}{2\lambda_\psi}T}.
\]
This finishes the proof. Note that the last claim follows since $\i$ is trivial in rank one.
\end{proof}

\section{Shadow estimates at parabolic limit points}
\label{sec.atparabolic}

Let $\Ga<G$ be $\theta$-Morse relative to $\cal P$, with Morse embedding
$f:Y\to X$ of a Gromov model $(Y,d)$ for $(\Ga,\cal P)$. 
Let
$\psi\in\fa_\theta^*$ be positive on $\L_f-\{0\}$.  We normalize
$\psi$ so that $\delta_\psi(\Ga)=1$.  
Then the existence and uniqueness of the Patterson--Sullivan measure is a consequence of the work of Canary--Zhang--Zimmer.

\begin{theorem} \cite{CZZ_relative} \label{thm.CZZ_PS}
    There exists a unique $(\Ga, \psi)$-Patterson--Sullivan measure $\nu$ on $\La_{\theta}$. Moreover, $\nu$ is atomless.
\end{theorem}

Let $\nu$ be a $(\Ga,\psi)$-Patterson--Sullivan measure on
$\La_\theta$ given in Theorem \ref{thm.CZZ_PS}. In this section we prove the global shadow lemma for
shadows along geodesic rays towards parabolic limit points and provide estimates on their complements.  Since all shadows we consider are
taken in $Y\cup\partial Y$, we identify $\nu$ with its pullback to
$\partial Y$ under the $\Ga$-equivariant homeomorphism
$f:\partial Y\to~\La_\theta$.

Suppose $\theta = \i(\theta)$ and set
\[
    \bar\psi:=\frac{\psi+\psi\circ\i}{2}.
\]
Since $\psi > 0$ on $\L_f - \{0\}$, so is
$\bar \psi$, and hence there exists $C_{\bar\psi}>0$ such that
\begin{equation} \label{eqn.C_barpsi}
    \bar\psi(v)>1-C_{\bar\psi}
    \quad
    \text{for all } v\in \mu(f(Y)^{-1}f(Y)).
\end{equation}

\subsection*{Shadows at parabolic limit points}

We first estimate shadows along geodesic rays ending at parabolic limit
points.  Recall that $o_Y\in Y$ is chosen so that $f(o_Y)=o$.  For
$\xi\in\partial Y$ and $t\ge 0$, let $\xi_t\in[o_Y,\xi]$ denote the point
with $d(o_Y,\xi_t)=t$, after choosing a geodesic ray $[o_Y,\xi]$.

\begin{theorem} \label{thm.shadowatpararep}
Let $\mathsf P\in\cal P$ and let $\xi=\xi_{\mathsf P}$.  For all
sufficiently large $R>0$,
\[
    \nu(O_R(o_Y,\xi_t))
    \asymp
    e^{2(\delta_{\bar\psi}(\mathsf P)-1)\d_{\bar\psi}(o_Y,\xi_t)}
    \bigl(C_{\bar\psi}+\d_{\bar\psi}(o_Y,\xi_t)\bigr)^{a_{\bar\psi}(\mathsf P)}
\]
for all $t\ge 0$, with implied constants independent of $t$.
\end{theorem}

For an arbitrary parabolic limit point, we obtain the following translated
form.

\begin{theorem} \label{thm.atparabolicgamma}
There exist constants $c,R_0>0$ with the following property.  Let
$\mathsf P\in\cal P$ and let $\xi=\ga\xi_{\mathsf P}$ for some
$\ga\in\Ga$.  Suppose that
\[
    d(\xi_{t_0},\ga o_Y)<c \quad \text{for some $t_0\ge 0$}.
\]
  Then, for every $C\ge C_{\bar\psi}$, 
$R>R_0$, and $t\ge t_0$,
\[
\begin{aligned}
    \nu(O_R(o_Y,\xi_t))
    \asymp{}
    e^{-\psi(\mu(\ga))}
    e^{2(\delta_{\bar\psi}(\mathsf P)-1)\d_{\bar\psi}(\ga o_Y,\xi_t)}
    \cdot
    \bigl(C+\d_{\bar\psi}(\ga o_Y,\xi_t)\bigr)^{a_{\bar\psi}(\mathsf P)} .
\end{aligned}
\]
\end{theorem}

For $g\in G$, define the translated measure $\nu_g$ by
\[
    d\nu_g(\eta)
    =
    e^{\psi(\beta_\eta^\theta(e,g))}\,d\nu(\eta).
\]
Then, for $g,h\in G$,
\[
    d\nu_g(\eta)
    =
    e^{\psi(\beta_\eta^\theta(h,g))}\,d\nu_h(\eta).
\]
For $x\in Y$, we write
\[
    \nu_x:=\nu_g
\]
where $g\in G$ satisfies $f(x)=go$.  This is independent of the choice of
$g$.

The next lemma relates $\nu$ to its translate at the center of a shadow.

\begin{lemma} \label{lem.changebasepoint}
For each fixed $R>0$, uniformly for
 $\xi\in\partial Y$ and $t\ge 0$,
\[
    \nu(O_R(o_Y,\xi_t))
    \asymp
    e^{-\d_\psi(o_Y,\xi_t)}
    \nu_{\xi_t}(O_R(o_Y,\xi_t)).
\]
\end{lemma}

Lemma \ref{lem.changebasepoint} is an immediate consequence of the
following comparison between Busemann maps and Cartan projections inside
shadows.  The corresponding statement for shadows in the symmetric space
$X=G/K$ was proved in \cite[Lemma 5.7]{LO_invariant}; the present version
follows from  comparing shadows in $Y$ with shadows in $X$ under the Morse
embedding $f$.

\begin{lemma} \label{lem.buseandcartan}
For every $R>0$, there exists $C>0$ such that, for all $x,y\in Y$ and $\xi\in O_R(x,y)$,
\[
    \left\|
        \beta_{f(\xi)}^\theta(g,h)-\mu_\theta(g^{-1}h)
    \right\|<C
\]
whenever $g,h\in G$ satisfy $go=f(x)$ and $ho=f(y)$.
\end{lemma}

Since $\Ga$ acts cocompactly on $Y-\bigcup\cal B$, Lemma
\ref{lem.changebasepoint} and the ordinary shadow lemma imply the
following thick-part estimate.

\begin{lemma}
There exists $R_0>0$ such that, for every $R>R_0$ and every
$x\in Y-\bigcup\cal B$,
\[
    \nu_x(O_R(o_Y,x))\asymp 1.
\]
\end{lemma}

The main estimate needed for Theorems \ref{thm.shadowatpararep} and
\ref{thm.atparabolicgamma} is the following.

\begin{proposition} \label{prop.shadowatparabolic_pre}
Let $\mathsf P\in\cal P$ and let $\xi=\xi_{\mathsf P}$.  For all
sufficiently large $R>0$, there exists $C>0$ such that, for all $t\ge 0$,
\[
\begin{aligned}
    \nu_{\xi_t}(O_R(o_Y,\xi_t))
    &\ll
    e^{\d_\psi(o_Y,\xi_t)}
    \sum_{\substack{g\in\mathsf P\\
    \bar\psi(\mu(g))\ge 2\d_{\bar\psi}(o_Y,\xi_t)-C}}
    e^{-\psi(\mu(g))},
    \\
    \nu_{\xi_t}(O_R(o_Y,\xi_t))
    &\gg
    e^{\d_\psi(o_Y,\xi_t)}
    \sum_{\substack{g\in\mathsf P\\
    \bar\psi(\mu(g))\ge 2\d_{\bar\psi}(o_Y,\xi_t)+C}}
    e^{-\psi(\mu(g))}.
\end{aligned}
\]
\end{proposition}

\begin{proof}
Choose $R>0$ large enough so that the ordinary shadow lemma holds.  Let
$Q\subset\partial Y-\{\xi\}$ be compact such that
$\mathsf P Q=\partial Y-\{\xi\}$ and such that the conclusion of
Lemma \ref{lem.shadowincpt} holds for the chosen $R$.  Increasing $R$ if
necessary, Proposition \ref{prop.shadowcompare} gives constants
$c_1,c_2>0$ such that
\[
\begin{aligned}
    \eta\in O_R(o_Y,\xi_t) - \{ \xi\}
    &\Longrightarrow
    \psi(\cal G^\theta(f(\xi),f(\eta)))
    \ge \d_{\bar\psi}(o_Y,\xi_t)-c_1,
    \\
    \psi(\cal G^\theta(f(\xi),f(\eta)))
    \ge \d_{\bar\psi}(o_Y,\xi_t)+c_2
    &\Longrightarrow
    \eta\in O_R(o_Y,\xi_t).
\end{aligned}
\]
By Proposition \ref{prop.gromovproductalongparabolic}, and using that
$\mu_\theta(g)-\i\mu_\theta(g)$ is uniformly bounded for
$g\in\mathsf P$, there exists $c>0$ such that
\begin{equation} \label{eqn.coveringshadowbycompacts}
\begin{aligned}
    \bigcup_{\substack{g\in\mathsf P\\
    \frac{1}{2}\bar\psi(\mu(g))
    \ge \d_{\bar\psi}(o_Y,\xi_t)+c+c_2}}
    gQ
    \subset
    O_R(o_Y,\xi_t) - \{\xi\}
    \subset
    \bigcup_{\substack{g\in\mathsf P\\
    \frac{1}{2}\bar\psi(\mu(g))
    \ge \d_{\bar\psi}(o_Y,\xi_t)-c-c_1}}
    gQ .
\end{aligned}
\end{equation}

If $gQ\subset O_R(o_Y,\xi_t)$, then Lemma \ref{lem.buseandcartan} gives
\[
    \nu_{\xi_t}(gQ)
    \asymp
    e^{\d_\psi(o_Y,\xi_t)}\nu(gQ).
\]
Similarly, if $gQ\cap O_R(o_Y,\xi_t)\ne\emptyset$, then
\[
    \nu_{\xi_t}(gQ\cap O_R(o_Y,\xi_t))
    \ll
    e^{\d_\psi(o_Y,\xi_t)}\nu(gQ).
\]
Since $\nu$ is atomless (Theorem \ref{thm.CZZ_PS}), using \eqref{eqn.coveringshadowbycompacts} and the bounded multiplicity of
the $\mathsf P$-translates of $Q$, we obtain 
\[
\begin{aligned}
    \nu_{\xi_t}(O_R(o_Y,\xi_t))
    &\ll
    e^{\d_\psi(o_Y,\xi_t)}
    \sum_{\substack{g\in\mathsf P\\
    \frac{1}{2}\bar\psi(\mu(g))
    \ge \d_{\bar\psi}(o_Y,\xi_t)-c-c_1}}
    \nu(gQ),
    \\
    \nu_{\xi_t}(O_R(o_Y,\xi_t))
    &\gg
    e^{\d_\psi(o_Y,\xi_t)}
    \sum_{\substack{g\in\mathsf P\\
    \frac{1}{2}\bar\psi(\mu(g))
    \ge \d_{\bar\psi}(o_Y,\xi_t)+c+c_2}}
    \nu(gQ).
\end{aligned}
\]
By Lemmas \ref{lem.shadowincpt} and \ref{lem.cptinshadow}, and the ordinary shadow
lemma,
\[
    \nu(gQ)\asymp e^{-\psi(\mu(g))}.
\]
The proposition follows.
\end{proof}

\subsection*{Proof of Theorem \ref{thm.shadowatpararep}}

Let $C>0$ be the constant from Proposition
\ref{prop.shadowatparabolic_pre}.  Let $k$ and $T_0$ be as in Theorem
\ref{thm.countinginsegments}, applied to $\bar\psi$.

First suppose that
\[
    2\d_{\bar\psi}(o_Y,\xi_t)-C>T_0.
\]
Set
\[
    T:=2\d_{\bar\psi}(o_Y,\xi_t)-C.
\]
By Proposition \ref{prop.gromovproductalongparabolic},
$\psi(\mu(g))\approx\bar\psi(\mu(g))$ for $g\in\mathsf P$. Hence $\delta_{\bar\psi}(\mathsf P)=\delta_\psi(\mathsf P)<1$
by Theorem~\ref{thm.entropydrop}.
Together with the entropy gap (Theorem \ref{thm.entropydrop}), this implies
$\delta_{\bar\psi}(\mathsf P)<1$.
Hence
Theorem \ref{thm.countinginsegments} gives
\[
\begin{aligned}
& 
  \sum_{\substack{g\in\mathsf P\\
  \bar\psi(\mu(g))\ge T}}
  e^{-\psi(\mu(g))} \\
 & \qquad \ll
\sum_{n=0}^{\infty}
e^{-(T+kn)}
\#\left\{
g\in\mathsf P:
T+kn\le \bar\psi(\mu(g))<T+k(n+1)
\right\}
\\
&\qquad \ll
\sum_{n=0}^{\infty}
e^{(\delta_{\bar\psi}(\mathsf P)-1)(T+kn)}
(1+T+kn)^{a_{\bar\psi}(\mathsf P)}
\\
&\qquad \asymp
e^{(\delta_{\bar\psi}(\mathsf P)-1)T}
(1+T)^{a_{\bar\psi}(\mathsf P)}
\\
&\qquad \asymp
e^{2(\delta_{\bar\psi}(\mathsf P)-1)\d_{\bar\psi}(o_Y,\xi_t)}
\bigl(C_{\bar\psi}+\d_{\bar\psi}(o_Y,\xi_t)\bigr)^{a_{\bar\psi}(\mathsf P)}.
\end{aligned}
\]
The lower bound is proved in the same way. 
Thus by Proposition \ref{prop.shadowatparabolic_pre},
\[
\begin{aligned}
    \nu_{\xi_t}(O_R(o_Y,\xi_t))
    \asymp{}
    e^{\d_\psi(o_Y,\xi_t)}
    e^{2(\delta_{\bar\psi}(\mathsf P)-1)\d_{\bar\psi}(o_Y,\xi_t)}
    \cdot
    \bigl(C_{\bar\psi}+\d_{\bar\psi}(o_Y,\xi_t)\bigr)^{a_{\bar\psi}(\mathsf P)} .
\end{aligned}
\]
Applying Lemma \ref{lem.changebasepoint} proves the theorem in this case.

It remains to consider the case when $2\d_{\bar\psi}(o_Y,\xi_t)-C\le T_0$.
Since $\bar\psi$ is positive on $\L_f-\{0\}$, this implies that
$d(o_Y,\xi_t)$ is uniformly bounded.  Hence, after increasing $R$ by a
uniform amount, $O_R(o_Y,\xi_t)=\partial Y$, and the desired estimate is
trivial.
\qed

\subsection*{Proof of Theorem \ref{thm.atparabolicgamma}}

Let $\xi=\ga_0\xi_{\mathsf P}$ for some $\ga_0\in\Ga$, and let
$B_\xi\in\cal B$ be the horoball based at $\xi$.  Then
$\stab_\Ga(B_\xi)=\ga_0\mathsf P\ga_0^{-1}$.  Since
$\ga_0\mathsf P\ga_0^{-1}$ acts cocompactly on $\partial B_\xi$, there
exist a uniform constant $c>0$, an element $\ga_1\in\mathsf P$, and
$t_0\ge 0$ such that
\[
    d(\xi_{t_0},\ga_0\ga_1o_Y)<c.
\]
Set $\ga:=\ga_0\ga_1$.

For $s\ge 0$, let $\eta_s\in[\ga o_Y,\xi]$ be the point satisfying
$d(\ga o_Y,\eta_s)=s$.  Increasing $c$ by a uniform amount, if necessary,
we have
\[
    d(\xi_t,\eta_{t-t_0})<c
    \quad\text{for all } t\ge t_0.
\]
Thus, there exists $c' > 0$ so that for all sufficiently large $R$,
\begin{equation} \label{eqn.linecut}
    O_{R-c'}(\ga o_Y,\eta_{t-t_0})
    \subset
    O_R(o_Y,\xi_t)
    \subset
    O_{R+c'}(\ga o_Y,\eta_{t-t_0})
\end{equation}
for all large $t \ge t_0$.

Moreover,
\[
    O_R(o_Y,\xi_t)\subset O_{R+c}(o_Y,\ga o_Y).
\]
Using Lemma \ref{lem.buseandcartan}, we obtain
\[
\begin{aligned}
    \nu(O_R(o_Y,\xi_t))
    &=
    \nu\bigl(\ga\,\ga^{-1}O_R(o_Y,\xi_t)\bigr)
    \\
    &=
    \int_{\ga^{-1}O_R(o_Y,\xi_t)}
    e^{\psi(\beta_x^\theta(e,\ga^{-1}))}\,d\nu(x)
    \\
    &=
    \int_{\ga^{-1}O_R(o_Y,\xi_t)}
    e^{-\psi(\beta_{\ga x}^\theta(e,\ga))}\,d\nu(x)
    \\
    &\asymp
    e^{-\psi(\mu(\ga))}
    \nu(\ga^{-1}O_R(o_Y,\xi_t)).
\end{aligned}
\]
By \eqref{eqn.linecut} and Theorem \ref{thm.shadowatpararep},
\[
\begin{aligned}
    \nu(O_R(o_Y,\xi_t))
    \asymp{}
    e^{-\psi(\mu(\ga))}
    e^{2(\delta_{\bar\psi}(\mathsf P)-1)
    \d_{\bar\psi}(\ga o_Y,\eta_{t-t_0})}
    \cdot
    \bigl(C_{\bar\psi}
    +\d_{\bar\psi}(\ga o_Y,\eta_{t-t_0})\bigr)^{a_{\bar\psi}(\mathsf P)} .
\end{aligned}
\]
Since \eqref{eqn.C_barpsi} holds, the factor
$C_{\bar\psi}+\d_{\bar\psi}$ may be replaced, up to multiplicative
constants, by $C+\d_{\bar\psi}$ for any $C\ge C_{\bar\psi}$.  Finally,
$\xi_t$ and $\eta_{t-t_0}$ are uniformly close, so Lemma
\ref{lem.cptcartan} gives
\[
    \d_{\bar\psi}(\ga o_Y,\eta_{t-t_0})
    \approx
    \d_{\bar\psi}(\ga o_Y,\xi_t).
\]
This proves the theorem.
\qed

\subsection*{Complements of parabolic shadows}

We also need an estimate for the complement of a shadow based at a
parabolic limit point.

\begin{proposition} \label{prop.complementshadow}
Let $\xi=\xi_{\mathsf P}$ for some $\mathsf P\in\cal P$.  For all
sufficiently large $R>0$ and all sufficiently large $t\ge 0$,
\[
\begin{aligned}
    \nu_{\xi_t}(\partial Y- O_R(o_Y,\xi_t))
    \asymp{}
    e^{-\d_\psi(\xi_t,o_Y)}
    e^{2\delta_{\bar\psi}(\mathsf P)\d_{\bar\psi}(o_Y,\xi_t)}
  \cdot
    \bigl(C_{\bar\psi}
    +\d_{\bar\psi}(o_Y,\xi_t)\bigr)^{a_{\bar\psi}(\mathsf P)} .
\end{aligned}
\]
\end{proposition}

\begin{proof}
Let $R>0$ and $Q\subset\partial Y-\{\xi\}$ be as in the proof of
Proposition \ref{prop.shadowatparabolic_pre}. 
For every $g\in\mathsf P$ and $\eta\in gQ$, Proposition
\ref{prop.gromovproductalongparabolic} gives
\[
    \left|\psi(\cal G^\theta(f(\xi),f(\eta)))
    -\tfrac12\bar\psi(\mu(g))\right|\le c.
\]
Together with Proposition~\ref{prop.shadowcompare} and
$\mathsf P Q=\partial Y-\{\xi\}$, this implies
\[
\begin{aligned}
    \bigcup_{\substack{g\in\mathsf P\\
    \frac{1}{2}\bar\psi(\mu(g))
    <\d_{\bar\psi}(o_Y,\xi_t)-c-c_1}}
    gQ
    \subset
    \partial Y- O_R(o_Y,\xi_t)
    \subset
    \bigcup_{\substack{g\in\mathsf P\\
    \frac{1}{2}\bar\psi(\mu(g))
    <\d_{\bar\psi}(o_Y,\xi_t)+c+c_2}}
    gQ .
\end{aligned}
\]
The thin-triangle property implies that there exists $R'>0$ such that, if
\[
    gQ\cap(\partial Y- O_R(o_Y,\xi_t))\ne\emptyset,
\]
then
\[
    d(\xi_t,g\xi_t)<R'.
\]
For such $g\in\mathsf P$, we therefore have
\[
    \nu_{\xi_t}(gQ)\asymp \nu_{\xi_t}(Q).
\]
After increasing $R$ by a uniform amount, we may also assume that
\[
    Q\subset O_R(\xi_t,o_Y).
\]
Hence Lemma \ref{lem.buseandcartan} gives
\[
    \nu_{\xi_t}(Q)
    \asymp
    e^{-\d_\psi(\xi_t,o_Y)}\nu(Q).
\]
Using Proposition \ref{prop.countingperipheral}, we obtain
\[
    \nu_{\xi_t}(\partial Y- O_R(o_Y,\xi_t))
    \asymp
    e^{-\d_\psi(\xi_t,o_Y)}
    e^{2\delta_{\bar\psi}(\mathsf P)\d_{\bar\psi}(o_Y,\xi_t)} 
    \bigl(C_{\bar\psi}
    +\d_{\bar\psi}(o_Y,\xi_t)\bigr)^{a_{\bar\psi}(\mathsf P)} 
\]
as claimed.
\end{proof}
\section{General form of the global shadow lemma}

Continuing with the setting of section \ref{sec.atparabolic}, we now prove
the general form of the global shadow lemma.  The estimate below should be
viewed as the cusp version of the usual shadow lemma: when the point under
consideration stays outside the horoballs in $\cal B$, one recovers the ordinary orbit-shadow estimate (Theorem \ref{thm.shadowlemmahigherrank}).

Recall that, for $\xi\in\partial Y$ and $t\ge 0$, we write
$\xi_t\in[o_Y,\xi]$ for the point satisfying $d(o_Y,\xi_t)=t$.

\begin{theorem}[Global Shadow Lemma] \label{thm.generalglobal}
For all sufficiently large $R>0$, the following estimate holds uniformly.
Let $\xi\in\partial Y$ and suppose that
$\xi_t\in B_\eta$
for some $t\ge 0$ and some horoball $B_\eta\in\cal B$ based at
$\eta\in\Ga\xi_{\mathsf P}$, where $\mathsf P\in\cal P$.  Let
$\ga\in\Ga$ be such that $\ga o_Y$ is a closest orbit point to $\xi_t$ in the orbit $\Gamma o_Y$, with respect to the metric $d$ on $Y$.
Then
\[
\begin{aligned}
    \nu(O_R(o_Y,\xi_t))
    \asymp{} &
    e^{-\d_\psi(o_Y,\xi_t)}
    e^{\d_\psi(\ga o_Y,\xi_t)}
    e^{2(\delta_{\bar\psi}(\mathsf P)-1)
        \d_{\bar\psi}(\ga o_Y,\xi_t)}
    \\
    &\cdot
    \bigl(C_{\bar\psi}
    +\d_{\bar\psi}(\ga o_Y,\xi_t)\bigr)^{a_{\bar\psi}(\mathsf P)} .
\end{aligned}
\]
Equivalently,
\[
\begin{aligned}
    \nu(O_R(o_Y,\xi_t))
    \asymp{}&
    e^{-\d_\psi(o_Y,\xi_t)}
    e^{\d_\psi(\Ga o_Y,\xi_t)}
    e^{2(\delta_{\bar\psi}(\mathsf P)-1)
        \d_{\bar\psi}(\Ga o_Y,\xi_t)}
    \\
    &\cdot
    \bigl(C_{\bar\psi}
    +\d_{\bar\psi}(\Ga o_Y,\xi_t)\bigr)^{a_{\bar\psi}(\mathsf P)} .
\end{aligned}
\]
The implied constants are independent of $\xi$, $t$, $\eta$, and $\ga$.
\end{theorem}

\begin{proof}
Let $A>0$ be a sufficiently large constant depending only on the
hyperbolicity constant of $Y$ and on the uniform constants fixed above.  We
allow $A$ to increase finitely many times during the proof.

First note that the endpoint of
\[
    [o_Y,\xi]\cap \partial B_\eta
\]
closest to $\xi_t$ is uniformly close to some orbit point of $\Ga o_Y$.
Consequently, if $\ga o_Y$ is chosen to be a closest orbit point to
$\xi_t$, for some $\ga \in \Ga$, then Lemma \ref{lem.closestinhoroball} gives
\[
    \d_\psi(\ga o_Y,\xi_t)
    \approx
    \d_\psi(\Ga o_Y,\xi_t)
    \quad \text{and} \quad 
    \d_{\bar\psi}(\ga o_Y,\xi_t)
    \approx
    \d_{\bar\psi}(\Ga o_Y,\xi_t).
\]
Thus it suffices to prove the estimate with this choice of $\ga$.

We divide the proof into three cases.

\medskip

\noindent
{\bf Case 1: $\eta\in O_{R-A}(o_Y,\xi_t)$.}
In this case, there exists $x\in[o_Y,\eta]$ such that
\[
    d(x,\xi_t)<R-A.
\]
This implies
\[
    O_A(o_Y,x)
    \subset
    O_R(o_Y,\xi_t)
    \subset
    O_{2R-A}(o_Y,x).
\]
Note that we may also choose $\ga \in \Ga$ so that $\ga o_Y$ is uniformly close to $[o_Y, \eta]$, in this case.
Applying Theorem \ref{thm.atparabolicgamma} to the parabolic point
$\eta$, and then using Lemma \ref{lem.cptcartan} to replace $x$ by
$\xi_t$, we obtain
\[
    \nu(O_R(o_Y,\xi_t))
    \asymp{}
    e^{-\psi(\mu(\ga))}
    e^{2(\delta_{\bar\psi}(\mathsf P)-1)
        \d_{\bar\psi}(\ga o_Y,\xi_t)} \cdot
    \bigl(C_{\bar\psi}
    +\d_{\bar\psi}(\ga o_Y,\xi_t)\bigr)^{a_{\bar\psi}(\mathsf P)} .
\]
Since $\xi_t$ lies in $B_\eta$ and is uniformly close to the ray
$[o_Y,\eta]$, Lemmas \ref{lem.closestinhoroball} and
\ref{lem.DKO_additive} imply
\[
    -\psi(\mu(\ga)) = - \d_{\psi}(o_Y, \ga o_Y)
    \approx
    -\d_\psi(o_Y,\xi_t)
    +
    \d_\psi(\ga o_Y,\xi_t).
\]
Substituting this into the preceding estimate proves the theorem in this
case.

\medskip

\noindent
{\bf Case 2: $\eta\notin O_{R+A}(o_Y,\xi_t)$.}
In particular, we have $\xi\ne\eta$.  Let $o'\in[o_Y,\xi]\cap\partial B_\eta$ be the
endpoint farthest from $o_Y$; equivalently, $o'$ is the point at which the
ray $[o_Y,\xi]$ exits the horoball $B_\eta$.  Set
\[
    t':=d(o',\xi_t).
\]

We first compare the shadow $O_R(o_Y,\xi_t)$ with complements of shadows
based at $o'$.  If $\zeta\in O_R(o_Y,\xi_t)$, then
$d([o_Y,\zeta],\xi_t)<R$, and hence
\[
    d(o',[o_Y,\zeta])
    \le
    d(o',\xi_t)+R.
\]
Therefore
\begin{equation} \label{eqn.aboveinclusionrel0}
    O_R(o_Y,\xi_t)
    \subset
    \partial Y
    -
    \left\{
        \zeta\in\partial Y:
        d(o',[o_Y,\zeta])>d(o',\xi_t)+R
    \right\}.
\end{equation}
Conversely, if
\[
    d(o',[o_Y,\zeta])
    \le
    d(o',\xi_t)+R- O(\delta), 
\]
then by the Gromov hyperbolicity, this implies that $[o_Y,\zeta]$ passes within distance $R$ of
$\xi_t$, provided 
$O(\delta)$ is sufficiently large.  Hence
\begin{equation} \label{eqn.aboveinclusionrel}
    \partial Y
    -
    \left\{
        \zeta\in\partial Y:
        d(o',[o_Y,\zeta])>d(o',\xi_t)+R-  O(\delta) 
    \right\}
    \subset
    O_R(o_Y,\xi_t).
\end{equation}

For $s\ge 0$, let $\eta_s\in[o',\eta]$ be the point satisfying
$d(o',\eta_s)=s$.  We claim that
\begin{equation} \label{eqn.shadowincomplement}
    O_R(o_Y,\xi_t)
    \subset
    \partial Y- O_{A/2}(o',\eta_{t'+R+A})
\end{equation}
and
\begin{equation} \label{eqn.complementinshadow}
    \partial Y- O_{A/2}(o',\eta_{t'+R})
    \subset
    O_R(o_Y,\xi_t).
\end{equation}

To prove \eqref{eqn.shadowincomplement}, let
\[
    \zeta\in O_{A/2}(o',\eta_{t'+R+A}).
\]
By Lemma \ref{lem.shadowandgromprod},
\[
    \langle \zeta,\eta\rangle_{o'}
    \ge
    t'+R+A/2-O(\delta).
\]
Since $\eta\notin O_{R+A}(o_Y,\xi_t)$, applying
\eqref{eqn.aboveinclusionrel} with $R+A$ in place of $R$ gives 
\[
    d(o',[o_Y,\eta])
    >
    t'+R+A-O(\delta).
\]
The Gromov product inequality \eqref{eqn.Gromovineq} then implies
\[
    \langle o_Y,\zeta\rangle_{o'}
    \ge
    t'+R+A/2-O(\delta).
\]
For $A$ sufficiently large, this and \eqref{eqn.aboveinclusionrel0} imply
that $\zeta\notin O_R(o_Y,\xi_t)$.  This proves
\eqref{eqn.shadowincomplement}.

The proof of \eqref{eqn.complementinshadow} is similar.  If
$\zeta\notin O_R(o_Y,\xi_t)$, then by
\eqref{eqn.aboveinclusionrel},
\[
    d(o',[o_Y,\zeta])>t'+R-O(\delta).
\]
Together with the estimate for $d(o',[o_Y,\eta])$ above, the Gromov
product inequality gives
\[
    \langle \zeta,\eta\rangle_{o'}
    \ge
    t'+R-O(\delta).
\]
By Lemma \ref{lem.shadowandgromprod}, this implies
\[
    \zeta\in O_{A/2}(o',\eta_{t'+R})
\]
after increasing $A$, if necessary.  This proves
\eqref{eqn.complementinshadow}.

\medskip

We now estimate the measure.  The hypothesis
$\eta\notin O_{R+A}(o_Y,\xi_t)$ implies, by thin triangles, that
$\eta_{t'}$ and $\xi_t$ are uniformly close.  Using
\eqref{eqn.shadowincomplement}, moving the basepoint from $\eta_{t'}$ to
$\eta_{t'+R+A}$, and absorbing the resulting multiplicative constant into
$\asymp$, we get
\[
\begin{aligned}
    \nu_{\xi_t}(O_R(o_Y,\xi_t))
    &\ll
    \nu_{\eta_{t'}}
    \bigl(\partial Y- O_{A/2}(o',\eta_{t'+R+A})\bigr)
    \\
    &\ll
    \nu_{\eta_{t'+R+A}}
    \bigl(\partial Y- O_{A/2}(o',\eta_{t'+R+A})\bigr).
\end{aligned}
\]

Since $o'\in\partial B_\eta$, there exists $\ga_0\in\Ga$ such that
$\eta=\ga_0\xi_{\mathsf P}$ and $d(o',\ga_0o_Y)$
is uniformly bounded.  By equivariance of the measures $\nu_x$,
\[
\begin{aligned}
&\nu_{\eta_{t'+R+A}}
\bigl(\partial Y- O_{A/2}(o',\eta_{t'+R+A})\bigr)
\\
&\qquad\asymp
\nu_{\ga_0^{-1}\eta_{t'+R+A}}
\bigl(
    \partial Y
    -
    O_{A/2}(o_Y,\ga_0^{-1}\eta_{t'+R+A})
\bigr).
\end{aligned}
\]
Applying Proposition \ref{prop.complementshadow}, we obtain
\[
\begin{aligned}
&\nu_{\ga_0^{-1}\eta_{t'+R+A}}
\bigl(
    \partial Y
    -
    O_{A/2}(o_Y,\ga_0^{-1}\eta_{t'+R+A})
\bigr)
\\
&\qquad\asymp
e^{-\d_\psi(\ga_0^{-1}\eta_{t'+R+A},o_Y)}
e^{2\delta_{\bar\psi}(\mathsf P)
    \d_{\bar\psi}(o_Y,\ga_0^{-1}\eta_{t'+R+A})}
\\
&\qquad\quad\cdot
\bigl(
    C_{\bar\psi}
    +
    \d_{\bar\psi}(o_Y,\ga_0^{-1}\eta_{t'+R+A})
\bigr)^{a_{\bar\psi}(\mathsf P)}
\\
&\qquad\asymp
e^{\d_\psi(\ga_0o_Y,\eta_{t'+R+A})}
e^{2(\delta_{\bar\psi}(\mathsf P)-1)
    \d_{\bar\psi}(\ga_0o_Y,\eta_{t'+R+A})}
\\
&\qquad\quad\cdot
\bigl(
    C_{\bar\psi}
    +
    \d_{\bar\psi}(\ga_0o_Y,\eta_{t'+R+A})
\bigr)^{a_{\bar\psi}(\mathsf P)} .
\end{aligned}
\]
Combining altogether gives
\[
\begin{aligned}
    \nu_{\xi_t}(O_R(o_Y,\xi_t))
    \ll{} &
    e^{\d_\psi(\ga_0o_Y,\eta_{t'+R+A})}
    e^{2(\delta_{\bar\psi}(\mathsf P)-1)
        \d_{\bar\psi}(\ga_0o_Y,\eta_{t'+R+A})}
    \\
    &\cdot
    \bigl(
        C_{\bar\psi}
        +
        \d_{\bar\psi}(\ga_0o_Y,\eta_{t'+R+A})
    \bigr)^{a_{\bar\psi}(\mathsf P)} .
\end{aligned}
\]
The points $\eta_{t'+R+A}$ and $\xi_t$ are at uniformly bounded distance
depending only on $R$ and $A$.  Therefore, by Lemmas
\ref{lem.cptcartan} and \ref{lem.closestinhoroball},
\[
    \d_\psi(\ga_0o_Y,\eta_{t'+R+A})
    \approx
    \d_\psi(\ga o_Y,\xi_t),
\]
and similarly for $\d_{\bar\psi}$.  Hence
\[
\begin{aligned}
    \nu_{\xi_t}(O_R(o_Y,\xi_t))
    \ll{}
    e^{\d_\psi(\ga o_Y,\xi_t)}
    e^{2(\delta_{\bar\psi}(\mathsf P)-1)
        \d_{\bar\psi}(\ga o_Y,\xi_t)}
    \cdot
    \bigl(
        C_{\bar\psi}
        +
        \d_{\bar\psi}(\ga o_Y,\xi_t)
    \bigr)^{a_{\bar\psi}(\mathsf P)} .
\end{aligned}
\]
The reverse inequality is obtained in the same way using 
\eqref{eqn.complementinshadow} instead of \eqref{eqn.shadowincomplement}.

Finally, Lemma \ref{lem.changebasepoint} gives
\[
    \nu(O_R(o_Y,\xi_t))
    \asymp
    e^{-\d_\psi(o_Y,\xi_t)}
    \nu_{\xi_t}(O_R(o_Y,\xi_t)),
\]
and the desired estimate follows.

\medskip

\noindent
{\bf Case 3: $\eta\notin O_{R-A}(o_Y,\xi_t)$ and
$\eta\in O_{R+A}(o_Y,\xi_t)$.} Take $R$ sufficiently large that $R-2A$ exceeds all the
preceding thresholds. Since
$\eta\notin O_{R-A}(o_Y,\xi_t)$, Case~2 applies to
$O_{R-2A}(o_Y,\xi_t)$. Since
$\eta\in O_{R+A}(o_Y,\xi_t)$, Case~1 applies to
$O_{R+2A}(o_Y,\xi_t)$.
Both estimates have the same right-hand side, up to uniform
multiplicative constants. The inclusions
\[
    O_{R-2A}(o_Y,\xi_t)
    \subset O_R(o_Y,\xi_t)
    \subset O_{R+2A}(o_Y,\xi_t)
\]
therefore give the desired estimate.

\medskip

The proof is complete.
\end{proof}

\begin{remark}\label{r1} When $\operatorname{rank} G=1$, the relatively Morse condition coincides
with geometric finiteness.  Moreover, after choosing the unit vector
$H_0\in\fa^+$, we identify $\fa$ with $\mathbb R H_0$ and take
$\psi(tH_0)=\delta_\Gamma t$.
Then
$\psi(\mu(g))=\delta_\Gamma d(o,go)$ and $\delta_\psi(\Gamma)=1$.
Since the opposition involution is trivial in rank one, $\bar\psi=\psi$; and
by Theorem \ref{thm.rankone_no_log}, we have
$ a_\psi(\mathsf P)=a_{\bar\psi}(\mathsf P)=0$.
When $G=\SO(n,1)^\circ$, Theorem~\ref{thm.generalglobal}
therefore specializes to Theorem~\ref{thm.SVglobal}.   In this sense, Theorem \ref{thm.generalglobal} is the
higher-rank relatively Morse analogue of the global shadow lemma of
Stratmann--Velani \cite{SV_globalshadow}.  The condition
$\delta_\psi(\Gamma)=1$ in Theorem \ref{thm.generalglobal} is merely a
normalization of the Patterson--Sullivan parameter; in the rank-one
specialization above, it is achieved by the choice
$\psi(tH_0)=\delta_\Gamma t$.\end{remark}

\section{Local properties of Patterson--Sullivan measures}

In this section, we apply the global shadow lemma to local properties of
Patterson--Sullivan measures.  
As before, let $\Ga<G$ be
$\theta$-Morse relative to $\cal P$, with Morse embedding
$f:Y\to X$ of a Gromov model $(Y,d)$ for $(\Ga,\cal P)$.  Let
$\psi\in\fa_\theta^*$ be such that
\[
    \psi>0
    \quad\text{on } \L_f-\{0\}
    \quad \text{and} \quad 
    \delta_\psi(\Ga)=1.
\]
Let $\nu$ be a $(\Ga,\psi)$-Patterson--Sullivan measure on
$\La_\theta$.  Throughout this section, we also assume that $\theta = \i(\theta)$.

\subsection*{The visual quasi-metric}
We first define the higher-rank visual quasi-metric associated to $\psi$.
For distinct $\xi,\eta\in\La_\theta$, set
\begin{equation} \label{eqn.visualmetric}
    d_\psi(\xi,\eta)
    :=
    e^{-\psi(\cal G^\theta(\xi,\eta))},
\end{equation}
and put $d_\psi(\xi,\xi)=0$.  Note that this is
not the same object as the metric-like function $\d_\psi$ on $Y$. This function behaves like a metric: there
exists $c>0$ such that
\begin{equation} \label{eqn.trianglevisual}
    d_\psi(\xi,\eta)
    \le
    c\bigl(d_\psi(\xi,\zeta)+d_\psi(\zeta,\eta)\bigr)
\end{equation}
for all $\xi,\eta,\zeta\in\La_\theta$.  For Anosov groups, this was proved
in \cite[Proposition 5.3]{LO_invariant}; the same argument applies to a
general Morse embedding.

For $r>0$ and $\xi\in\La_\theta$, let
\[
    B_\psi(\xi,r)
    :=
    \{\eta\in\La_\theta:d_\psi(\xi,\eta)<r\}.
\]

By Lemma \ref{lem.gromovproductnearest}, $d_\psi$ and
$d_{\bar\psi}$ are bi-Lipschitz equivalent, since
$\frac12(\mu_\theta(f(z))+\i\mu_\theta(f(z)))$
is $\i$-invariant and hence $\psi$ and $\bar\psi$ agree on it.
In particular, $B_\psi$-balls and $B_{\bar\psi}$-balls are comparable.

We identify $\partial Y$ with $\La_\theta$ via the
$\Ga$-equivariant homeomorphism $f:\partial Y\to\La_\theta$. 
Proposition \ref{prop.shadowcompare} implies that, for all sufficiently
large $R>0$, there exist constants $c_1,c_2>0$ such that
\begin{equation} \label{eqn.compareshadowsandballs}
    B_\psi\bigl(\xi,c_1 e^{-\d_{\bar \psi}(o_Y,\xi_t)}\bigr)
    \subset
    O_R(o_Y,\xi_t)
    \subset
    B_\psi\bigl(\xi,c_2 e^{-\d_{\bar \psi}(o_Y,\xi_t)}\bigr)
\end{equation}
for all $\xi\in\La_\theta$ and all $t\ge 0$, where
$\xi_t\in[o_Y,\xi]$ is the point satisfying $d(o_Y,\xi_t)=t$.

For later use, we rewrite the global shadow lemma in a compact form.
If $x\in Y$ lies in a horoball based at a point of
$\Ga\xi_{\mathsf P}$, choose $\gamma_x\in\Ga$ so that
$\gamma_x o_Y$ is a closest orbit point to $x$, and set
\[
    h_\psi(x):=\d_\psi(\gamma_x o_Y,x),\quad
    \delta(x):=\delta_\psi(\mathsf P),\quad \text{and} \quad
    a(x):=a_{\bar \psi}(\mathsf P).
\]
If $x$ lies outside the horoballs, we set
\[
    h_\psi(x)=h_{\bar{\psi}}(x)=0,\quad \delta(x)=0,\quad \text{and} \quad a(x)=0.
\]

By Proposition \ref{prop.gromovproductalongparabolic},
$\psi(\mu_\theta(\gamma))
=\bar\psi(\mu_\theta(\gamma))+O(1)$
uniformly for $\gamma\in\mathsf P$.
Hence the corresponding Poincar\'e series have the same
abscissa of convergence, so
$\delta_\psi(\mathsf P)=\delta_{\bar\psi}(\mathsf P)$ for every $\mathsf P\in\cal P$.
Define
$$
F(x):=\d_\psi(o_Y,x)-h_\psi(x)
      +2(1-\delta(x))h_{\bar\psi}(x).
$$
The global shadow lemma (Theorem \ref{thm.generalglobal})
then takes the form
\begin{equation}\label{eqn.uniformshadowlocal}
\nu(O_R(o_Y,x))
\asymp e^{-F(x)}
       (C_{\bar\psi}+h_{\bar\psi}(x))^{a(x)},
\end{equation}
with implied constants independent of $x$.

\begin{lemma} \label{lem.QI log shadow} For every sufficiently large $R>0$, there exists
$C=C(R)>1$ such that, for every $y\in Y$
and every $x\in[o_Y,y]$,
$$
C^{-1} d(x, y) - C \le \log \frac{\nu(O_R(o_Y, x))}{\nu(O_R(o_Y, y))} \le C d(x, y) + C.
$$
\end{lemma}

\begin{proof}
By tautness and hyperbolicity, there exists a uniform $D>0$
such that $x,y$ are within distance $D$ of ordered points
$x',y'$ on a geodesic ray from $o_Y$. Taking $R$ sufficiently
large, the inclusions
$$
O_{R-D}(o_Y,z')\subset O_R(o_Y,z)
\subset O_{R+D}(o_Y,z')
$$
and the global shadow lemma give
$\nu(O_R(o_Y,z))\asymp\nu(O_R(o_Y,z'))$ for $z=x,y$.
Since $|d(x,y)-d(x',y')|\le 2D$, it suffices to consider
$x,y$ on a common geodesic ray from $o_Y$.

First suppose that both $x$ and $y$ are uniformly close to
$\Ga o_Y$. Then $h_\psi$ and $h_{\bar\psi}$ are uniformly
bounded at both points.
By \eqref{eqn.uniformshadowlocal} and Lemma \ref{lem.DKO_additive},
$$
\log\frac{\nu(O_R(o_Y,x))}{\nu(O_R(o_Y,y))}
\approx \d_\psi(o_Y,y)-\d_\psi(o_Y,x)
\approx \d_\psi(x,y).
$$
The desired bounds follow from Lemma \ref{lem.QIpropsofmetriclike}.

Now suppose that $x$ and $y$ belong to the same horoball
based at $p$. If the geodesic ray from $o_Y$ containing
$x,y$ exits the horoball, let $u,v$ be its entrance and
exit points, and let $m$ be the midpoint of $[u,v]$.
Since $u,v$ lie on the boundary of the same horoball,
hyperbolicity implies that $[u,m]$ and $[m,v]$ stay
uniformly close to $[u,p]$ and $[v,p]$, respectively.
Applying Lemma \ref{lem.closestinhoroball} to $\psi\circ\i$,
after translating orbit points uniformly close to $u,v$,
and using Lemma \ref{lem.DKO_additive}, we obtain:
\begin{enumerate}
    \item If $y$ lies in the first half of $[u,v]$, then
    $$
    h_\psi(y)-h_\psi(x)\approx\d_\psi(x,y).
    $$
    \item If $x$ lies in the second half of $[u,v]$, then
    $$
    h_\psi(y)-h_\psi(x)\approx-\d_\psi(y,x).
    $$
\end{enumerate}
Here we used
$\d_{\psi\circ\i}(z,w)=\d_\psi(w,z)$.
The same statements hold with $\psi$ replaced by $\bar\psi$.
If the ray does not exit the horoball, it ends at $p$,
and the same argument gives (1) throughout its terminal
part in the horoball.

Write $\delta=\delta(x)=\delta(y)$ and $a=a(x)=a(y)$.
By Theorem \ref{thm.entropydrop} and the finiteness of
$\cal P$, both $\delta$ and $1-\delta$ are uniformly bounded
away from zero, and $a$ is uniformly bounded.

In case (1), Lemma \ref{lem.DKO_additive} gives
$$
\begin{aligned}
F(y)-F(x)
&\approx \d_\psi(x,y)
-\bigl(h_\psi(y)-h_\psi(x)\bigr)
+2(1-\delta)\bigl(h_{\bar\psi}(y)-h_{\bar\psi}(x)\bigr)\\
&\approx 2(1-\delta)\d_{\bar\psi}(x,y).
\end{aligned}
$$
We also have
$$
\begin{aligned}
\left|
\log\frac{C_{\bar\psi}+h_{\bar\psi}(x)}
         {C_{\bar\psi}+h_{\bar\psi}(y)}
\right|
&\ll 1+\log\bigl(1+|h_{\bar\psi}(x)-h_{\bar\psi}(y)|\bigr)\\
&\ll 1+\log(1+d(x,y)).
\end{aligned}
$$
Since $\d_{\bar\psi}$ satisfies the quasi-isometric bounds
in Lemma \ref{lem.QIpropsofmetriclike}, the logarithmic
correction can be absorbed into the positive linear term,
using $\log(1+t)\le\varepsilon t+C_\varepsilon$ with
$\varepsilon>0$ sufficiently small. The claim therefore
follows from \eqref{eqn.uniformshadowlocal}.

In case (2), we similarly have
$$
F(y)-F(x)\approx 2\delta\,\d_{\bar\psi}(x,y),
$$
using
$\d_\psi(x,y)+\d_\psi(y,x)=2\d_{\bar\psi}(x,y)$.
Together with the preceding bound on the logarithmic
correction, the same argument proves the claim in this~case.

If $x,y$ lie in the same horoball and $m\in[x,y]$, apply
the bounds proved in case (1) to $(x,m)$ and those proved
in case (2) to $(m,y)$. Adding them gives the desired
bounds for $(x,y)$.

Finally, suppose that $x,y$ lie in different horoballs,
allowing either point to lie outside the horoballs.
Let $x_1$ be the exit point from the horoball containing
$x$, taking $x_1=x$ if $x$ lies outside the horoballs.
Similarly, let $y_1$ be the entrance point into the horoball
containing $y$, taking $y_1=y$ if $y$ lies outside the
horoballs. Then
$
[x,y]=[x,x_1]\cup[x_1,y_1]\cup[y_1,y].
$
Each of the initial and terminal subsegments splits into
at most two pieces covered by (1) or (2). Both endpoints
of the middle subsegment are uniformly close to $\Ga o_Y$,
so the first case of the proof applies to it, regardless
of the horoballs it crosses. Thus there are at most
$5$ pieces. Summing their estimates proves the lemma.\end{proof}

\begin{remark} \label{rmk.ballestimate}
 When $\psi = \psi \circ \i$, using \eqref{eqn.compareshadowsandballs}, \eqref{eqn.uniformshadowlocal} can also be written in terms of $d_{\psi}$-balls: for $\xi \in \La_{\theta}$ and $r > 0$, let $x \in [o_Y, \xi]$ be such that $\d_{\psi}(o_Y, x) \approx - \log r$ with a uniform additive error. Then 
    $$
    \nu(B_{\psi}(\xi, r)) \asymp r e^{(2 \delta(x) - 1) h_{\psi}(x)} (C_{\psi} + h_{\psi}(x))^{a(x)}.
    $$
\end{remark}

\subsection*{Local doubling and reverse doubling}

We now deduce that $\nu$ is locally doubling and locally reverse doubling
with respect to $d_\psi$. The latter means that shrinking the
radius by a sufficiently large factor reduces the  mass by any prescribed
factor, uniformly in the center and the scale.

\begin{theorem} \label{thm.doubling}
\label{thm.localnonconcentration}

\begin{enumerate}
    \item For every $L>1$, there exists $\varepsilon_L>0$ such that
\[
    \nu(B_\psi(\xi,r))
    \ge
    \varepsilon_L\,\nu(B_\psi(\xi,Lr))
\]
for all $\xi\in\La_\theta$ and all $r>0$ with $Lr\le 1$.

\item 

For every $\kappa\ge 1$, there exists $L>1$ such that
\[
    \nu(B_\psi(\xi,r/L))
    \le
    \kappa^{-1}\nu(B_\psi(\xi,r))
\]
for all $\xi\in\La_\theta$ and all $0<r\le 1$.
\end{enumerate}

\end{theorem}

\begin{proof}
Coarse additivity and Lemma \ref{lem.QIpropsofmetriclike}
compare increments of $\d_{\bar\psi}(o_Y,\xi_t)$ with
distances along the ray, up to uniform multiplicative
and additive constants. Combining this with
Lemma \ref{lem.QI log shadow} and
\eqref{eqn.compareshadowsandballs}, we obtain constants
$C,c>1$ such that
$$
C^{-1}\left(\frac{r}{s}\right)^c
\le
\frac{\nu(B_\psi(\xi,r))}{\nu(B_\psi(\xi,s))}
\le
C\left(\frac{r}{s}\right)^{1/c}
$$
for every $\xi\in\La_\theta$ and $0<r\le s\le1$.

Taking $s=Lr$ proves (1) with
$\varepsilon_L=C^{-1}L^{-c}$.
For (2), choose $L\ge(C\kappa)^c$ and apply the upper bound
to the radii $r/L$ and $r$.
\end{proof}

\section{Hausdorff measure}

\label{sec.Hausdorff}

In this section, we characterize when the Patterson--Sullivan measure
agrees, up to scale, with the Hausdorff measure defined by the visual
quasi-metric. As in the preceding sections, let $\Ga<G$ be
$\theta$-Morse relative to $\cal P$, with Morse embedding
\[
    f:Y\to X
\]
from a Gromov model $(Y,d)$ for $(\Ga,\cal P)$. Let
$\psi\in\fa_\theta^*$ satisfy
\[
    \psi>0\quad\text{on }\L_f-\{0\},
    \qquad \delta_\psi(\Ga)=1.
\]

Let $\nu$ be the corresponding Patterson--Sullivan probability measure
on $\La_\theta$. We assume that $\theta=\i(\theta)$ and write
\[
    \bar\psi:=\frac{\psi+\psi\circ\i}{2}.
\]
We identify $\partial Y$ with $\La_\theta$ via $f$.

Equip $\La_\theta$ with the visual quasi-metric $d_\psi$ defined in
\eqref{eqn.visualmetric}. The quasi-triangle inequality in \eqref{eqn.trianglevisual}, together
with $d_\psi\asymp d_{\bar\psi}$ and the symmetry of $d_{\bar\psi}$,
gives the usual covering lemma:
from a family of balls with uniformly bounded radii, one can select a
countable disjoint subfamily whose expansions by a uniform factor cover
the centers of the original balls; see, for instance,
\cite[Lemma 6.12]{LO_invariant}. We use this form of the Vitali covering
lemma below.

For $s>0$, $\varepsilon>0$, and $E\subset\La_\theta$, define
\[
    \cal H_{\psi,\varepsilon}^s(E)
    :=\inf\left\{
        \sum_i(\diam_\psi U_i)^s:
        E\subset\bigcup_i U_i,\quad
        \sup_i\diam_\psi U_i\le\varepsilon
    \right\},
\]
where $\diam_\psi U:=\sup_{\xi,\eta\in U}d_\psi(\xi,\eta)$.
The outer measure
\[
    \cal H_\psi^s(E)
    :=\lim_{\varepsilon\to0}\cal H_{\psi,\varepsilon}^s(E)
\]
restricts to a Borel measure on $\La_\theta$; see
\cite{Falconer_fractatl} and \cite[Appendix A]{DK_patterson}.
We call it the $s$-dimensional Hausdorff measure associated to $d_\psi$.
For $s=1$, write $\cal H_{\psi,\varepsilon}:=\cal H_{\psi,\varepsilon}^1$
and $\cal H_\psi:=\cal H_\psi^1$.

\begin{theorem}\label{thm.Hmeasure}
If $\Ga$ is Zariski dense, then $\nu$ is a positive multiple of
$\cal H_\psi$ if and only if both of the following hold:
\begin{enumerate}
    \item $\psi=\psi\circ\i$;
    \item for every $\mathsf P\in\cal P$,
    \[
        \delta_\psi(\mathsf P)<\frac12,
        \quad\text{or}\quad
        \delta_\psi(\mathsf P)=\frac12
        \ \text{and}\ a_\psi(\mathsf P)=0.
    \]
\end{enumerate}
For arbitrary $\Ga$, if $\psi$ is symmetric, then $\nu$ is a
positive multiple of $\cal H_\psi$ if and only if condition~(2) holds.

\end{theorem}

\begin{remark}
This theorem generalizes Sullivan's Hausdorff-measure criterion for
geometrically finite Kleinian groups \cite{Sullivan1984entropy}.
In the real hyperbolic case, a rank-$k$ parabolic subgroup has critical
exponent $\delta_{\mathsf P}=k/2$. After the normalization
$\delta_\psi(\Ga)=1$, we have $\delta_\psi(\mathsf P)
    =\frac{\delta_{\mathsf P}}{\delta_\Ga}$
and $a_\psi(\mathsf P)=0$. Thus the condition in the theorem becomes
$k\le\delta_\Ga$ for every cusp.
Anosov groups are the case with no peripheral subgroups, so the theorem
also generalizes \cite[Theorem 1.1]{DKO_AR}.
\end{remark}

If $\psi$ is symmetric and
$0<\cal H_\psi(\La_\theta)<\infty$, the Hausdorff measure has
the same conformality rule as $\nu$:

\[
    \frac{d\gamma_*\cal H_\psi}{d\cal H_\psi}(\xi)
    =e^{\psi(\beta_\xi^\theta(e,\gamma))}
    \qquad(\gamma\in\Ga).
\]
The proof of \cite[Lemma 9.7]{DKO_AR}, stated there for Anosov groups,
applies unchanged, since it uses the transformation rule for $d_\psi$.
Uniqueness of the Patterson--Sullivan probability measure
(Theorem \ref{thm.CZZ_PS}) then gives
\begin{equation}\label{eqn.Hmeasureuniqueness}
    \cal H_\psi
    =\cal H_\psi(\La_\theta)\,\nu.
\end{equation}

We use this identity only for symmetric $\psi$ with positive and
finite total Hausdorff measure.

\subsection*{Local estimates}

Denote by $\La_\theta^{\rm con}$ the image of the conical limit set in
$\partial Y$. We first establish the local estimates for symmetric
$\psi$. The general finiteness statement will then follow by applying
these estimates to a normalized multiple of $\bar\psi$.

\begin{lemma}\label{lem.Hmeasurelocalsize}
Assume that $\psi=\psi\circ\i$.
There exists $C>1$ such that:
\begin{enumerate}
    \item For every $\xi\in\La_\theta^{\rm con}$, there exists a sequence
    $r_i\to0$ such that
    \[
        \nu(B_\psi(\xi,r_i))\ge C^{-1}r_i
        \qquad\text{for all }i.
    \]
    \item If the peripheral condition in Theorem \ref{thm.Hmeasure}
    holds, then, after increasing $C$ if necessary,
    \[
        \nu(B_\psi(\xi,r))\le Cr
        \qquad(\xi\in\La_\theta,\ r>0).
    \]
\end{enumerate}
\end{lemma}
\begin{proof}
Fix $R>0$ sufficiently large for
\eqref{eqn.compareshadowsandballs} and
\eqref{eqn.uniformshadowlocal} to hold. Since $\psi=\bar\psi$, the
latter estimate becomes
\begin{equation}\label{eqn.Hmeasurecusp}
    \nu(O_R(o_Y,x))
    \asymp e^{-\d_\psi(o_Y,x)}
    e^{(2\delta(x)-1)h_\psi(x)}
    (C_\psi+h_\psi(x))^{a(x)}.
\end{equation}
All implied constants are independent of $x$.

For (1), fix $\xi\in\La_\theta^{\rm con}$. The ray $[o_Y,\xi]$ meets
$Y-\bigcup_{B\in\mathcal B}B$ at arbitrarily large times.
Indeed, otherwise a tail of the ray would lie in a single horoball,
so $\xi$ would be its parabolic basepoint, contrary to conicality.
Choose $x_i\in[o_Y,\xi]-\bigcup_{B\in\mathcal B}B$ with
$d(o_Y,x_i)\to\infty$. These points are uniformly close to
$\Ga o_Y$ by cocompactness of the thick part, and $h_\psi(x_i)=0$.
Hence \eqref{eqn.Hmeasurecusp} gives
\[
    \nu(O_R(o_Y,x_i))\asymp e^{-\d_\psi(o_Y,x_i)}
\]
with constants independent of $\xi$ and $i$.
By \eqref{eqn.compareshadowsandballs}, there is a uniform $c_2>0$ such
that, for $r_i:=c_2e^{-\d_\psi(o_Y,x_i)}$,
\[
    O_R(o_Y,x_i)\subset B_\psi(\xi,r_i).
\]
Lemma \ref{lem.QIpropsofmetriclike} gives $r_i\to0$, proving (1).

To prove (2), note that by the hypothesis on the parabolic subgroups,
the cusp correction factor in the global shadow lemma is uniformly
bounded above. Indeed, if $x$ lies in a horoball associated to
$\mathsf P$, then the correction factor is
\[
    e^{(2\delta_\psi(\mathsf P)-1)h_\psi(x)}
    \bigl(C_\psi+h_\psi(x)\bigr)^{a_\psi(\mathsf P)}.
\]
This is uniformly bounded when
$\delta_\psi(\mathsf P)<1/2$, and also when
$\delta_\psi(\mathsf P)=1/2$ and $a_\psi(\mathsf P)=0$.
Here we use that $h_\psi$ is uniformly bounded below and
$\cal P$ is finite. In the thick part, the usual shadow estimate
gives the same conclusion. Hence
\[
    \nu(O_R(o_Y,x))
    \ll
    e^{-\d_\psi(o_Y,x)}
    \quad\text{uniformly for all $x\in Y$.}
\]

Now fix $\xi\in\La_\theta$ and let $r>0$ be sufficiently small.
Let $c_1$ be the constant in
\eqref{eqn.compareshadowsandballs}, and let $A$ be the constant
in Lemma \ref{lem.QIpropsofmetriclike}(3).
Apply that lemma with $s=\log(c_1/r)-A$ to choose
$x\in[o_Y,\xi]$ such that
\[
    \log(c_1/r)-2A
    \le \d_\psi(o_Y,x)
    \le \log(c_1/r).
\]
Thus
 $r\le c_1e^{-\d_\psi(o_Y,x)}\ll r$,
and \eqref{eqn.compareshadowsandballs} gives
$B_\psi(\xi,r)\subset O_R(o_Y,x)$.
Then
\[
    \nu(B_\psi(\xi,r))
    \le
    \nu(O_R(o_Y,x))
    \ll
    e^{-\d_\psi(o_Y,x)}
    \asymp r.
\]
After increasing the constant, the same bound holds for all $r>0$,
since $\nu$ is a probability measure.
\end{proof}
\subsection*{Finiteness and positivity}

\begin{proposition}\label{prop.Hmeasurepositivefinite}
The total Hausdorff measure satisfies
\[
    \cal H_\psi(\La_\theta)<\infty.
\]
It is positive if and only if $\delta_{\bar\psi}(\Ga)=1$ and,
for every $\mathsf P\in\cal P$,
\[
    \delta_{\bar\psi}(\mathsf P)<\frac12,
    \quad\text{or}\quad
    \delta_{\bar\psi}(\mathsf P)=\frac12
    \ \text{and}\ a_{\bar\psi}(\mathsf P)=0.
\]
If $\Ga$ is Zariski dense, these conditions are equivalent to
conditions~(1) and~(2) of Theorem~\ref{thm.Hmeasure}.
In particular, for Zariski dense $\Ga$ and nonsymmetric $\psi$,
$\cal H_\psi(\La_\theta)=0$.

\end{proposition}

\begin{proof}
We first prove the proposition when $\psi$ is symmetric.
Every limit point is either conical or parabolic.
The parabolic limit set is countable and
therefore has $\cal H_\psi$-measure zero. It therefore suffices to
show that
$\cal H_\psi(\La_\theta^{\rm con})<\infty$. 
Fix $\varepsilon>0$. For each $\xi\in\La_\theta^{\rm con}$,
choose $0<r_\xi<\varepsilon$ such that
$\nu(B_\psi(\xi,r_\xi))\ge C^{-1}r_\xi$, using
Lemma \ref{lem.Hmeasurelocalsize}(1).

Applying the Vitali covering lemma to the family
$\{B_\psi(\xi,r_\xi):\xi\in\La_\theta^{\rm con}\}$, 
there exists a countable disjoint subcollection $\{B_\psi(\xi_n,r_n):n\in\mathbb N\}$
such that $\La_\theta^{\rm con}
    \subset
    \bigcup_n B_\psi(\xi_n,\lambda r_n)$
for some uniform constant $\lambda>1$.  Since $d_\psi$ satisfies the
triangle inequality up to a multiplicative constant \eqref{eqn.trianglevisual}, there exists
$D>0$ such that
\[
    \diam_\psi B_\psi(\xi_n,\lambda r_n)
    \le
    D r_n \quad\text{for all $n$.}
\]
  Hence
\[
\begin{aligned}
    \cal H_{\psi,D\varepsilon}(\La_\theta^{\rm con})
    \le
    \sum_n D r_n  \le
    DC\sum_n \nu(B_\psi(\xi_n,r_n))  \le
    DC\,\nu(\La_\theta).
\end{aligned}
\]
Since $\varepsilon>0$ is arbitrary, this proves
$\cal H_\psi(\La_\theta^{\rm con})<\infty$.

Suppose next that the peripheral condition holds. Let $\{U_i\}$ be a
countable cover of $\La_\theta$ with $\diam_\psi U_i\le\varepsilon$.
For each nonempty $U_i$, choose $\xi_i\in U_i$ and
$\rho_i>\diam_\psi U_i$ so that
\[
    \sum_i\rho_i\le\varepsilon+\sum_i\diam_\psi U_i.
\]
Since $U_i\subset B_\psi(\xi_i,\rho_i)$,
Lemma \ref{lem.Hmeasurelocalsize}(2) gives
\[
    1\le\sum_i\nu(B_\psi(\xi_i,\rho_i))
    \le C\sum_i\rho_i
    \le C\left(\varepsilon+\sum_i\diam_\psi U_i\right).
\]
Taking the infimum and then letting $\varepsilon\to0$ proves
$\cal H_\psi(\La_\theta)>0$.

It remains to prove that a peripheral subgroup violating the condition in
Theorem \ref{thm.Hmeasure}
forces the Hausdorff measure to vanish. We first show that almost every
ray enters translates of any fixed peripheral horoball to arbitrarily
large depth.

Fix $\mathsf P\in\cal P$, let $U_{\mathsf P}$ be the union of all
$\Ga$-translates of its horoball, and set
\[
    b_{\mathsf P}(z):=d(z,Y -  U_{\mathsf P})
    \qquad(z\in Y).
\]
This is a $\Ga$-invariant $1$-Lipschitz function.

The diagonal $\Ga$-action on
$(\La_\theta^{(2)},\nu\otimes\nu)$ is ergodic by
\cite[Corollary 1.5]{CZZ_relative}. It follows that almost every pair
$(\xi,\eta)$ has a dense $\Ga$-orbit in $\La_\theta^{(2)}$.

Let $B_{\mathsf P}$ be the fixed horoball based at  $p=\xi_{\mathsf P}$.
We claim that for each integer $N>0$, there is a nonempty open set
$V_N\subset\La_\theta^{(2)}$ such that every geodesic joining
a pair in $V_N$ contains a point $z\in B_{\mathsf P}$ with
$b_{\mathsf P}(z)>N$.  To show this,
choose $q\in\partial Y-\{p\}$ and a geodesic $[p,q]$.
Let $D>0$ be a uniform constant for thin ideal quadrilaterals.
For each integer $N>0$, choose $z\in[p,q]$ sufficiently far
toward $p$ that
\[
    d(z,Y -  B_{\mathsf P})>N+D.
\]

Choose disjoint open neighborhoods $V_p,V_q$ of $p,q$
sufficiently small that the sides $[p,\xi]$ and $[q,\eta]$
stay farther than $D$ from $z$ whenever
$\xi\in V_p-\{p\}$ and $\eta\in V_q-\{q\}$.
Such neighborhoods exist because the Gromov products
$\langle p,\xi\rangle_z$ and $\langle q,\eta\rangle_z$
tend to infinity as $\xi\to p$ and $\eta\to q$.
Thinness of the ideal quadrilateral with vertices
$p,q,\eta,\xi$ then forces every geodesic $[\xi,\eta]$
to pass within distance $D$ of $z$.
Thus it contains a point $w$ satisfying
\[
   b_{\mathsf P}(w) \ge
    d(w,Y -  B_{\mathsf P})
    \ge d(z,Y -  B_{\mathsf P})-D>N.
\]

In the preceding argument, take $\xi\ne p$ and $\eta\ne q$
so that the four ideal vertices are distinct. Since $\partial Y$
has no isolated points,
$V_N:=(V_p-\{p\})\times(V_q-\{q\})$ is a nonempty open subset
of $\La_\theta^{(2)}$ with the required property.

For a pair $(\xi,\eta)$ with dense $\Ga$-orbit, choose
$\gamma_N\in\Ga$ such that $\gamma_N(\xi,\eta)\in V_N$.
Translating back and using the $\Ga$-invariance of $b_{\mathsf P}$
shows that every geodesic joining $\xi$ and $\eta$ contains a point $z_N$
where $b_{\mathsf P}(z_N)>N$. Thus, for almost every pair,
$b_{\mathsf P}$ is unbounded on every joining geodesic.

Let $E_{\mathsf P}$ be the set of endpoints $\xi$ for which
$b_{\mathsf P}$ is unbounded on $[o_Y,\xi]$.
This property does not depend on the choice of the ray, since
rays from $o_Y$ to the same endpoint stay uniformly close.
The set $E_{\mathsf P}$ is Borel: its complement is the countable
union, over $k\in\N$, of the compact sets of endpoints admitting
a ray from $o_Y$ on which $b_{\mathsf P}\le k$.
Here compactness follows from properness of $Y$ and continuity
of the endpoint map.

If $\nu(E_{\mathsf P}^c)>0$, then for distinct
$\xi,\eta\in E_{\mathsf P}^c$, thinness of the ideal triangle with
vertices $o_Y,\xi,\eta$ places every joining geodesic in a uniform
neighborhood of $[o_Y,\xi]\cup[o_Y,\eta]$.
Hence $b_{\mathsf P}$ is bounded on every such geodesic.
Since $\nu$ is atomless, the set of distinct pairs in
$E_{\mathsf P}^c\times E_{\mathsf P}^c$ has product measure
$\nu(E_{\mathsf P}^c)^2>0$.
For every such pair, $b_{\mathsf P}$ is bounded on joining
geodesics, whereas the preceding statement says that it is
unbounded for almost every distinct pair.
This contradiction shows that $\nu(E_{\mathsf P}^c)=0$, and hence
\[
    \nu(E_{\mathsf P})=1.
\]
Now suppose that $\mathsf P$ violates the peripheral condition in
Theorem \ref{thm.Hmeasure}.
Since $a_\psi(\mathsf P)$ is a nonnegative integer, either
$\delta_\psi(\mathsf P)>1/2$, or
$\delta_\psi(\mathsf P)=1/2$ and $a_\psi(\mathsf P)>0$.
For $\xi\in E_{\mathsf P}$, choose
$x_n\in[o_Y,\xi]\cap U_{\mathsf P}$ with
$b_{\mathsf P}(x_n)\to\infty$.
Since
\[
    d(x_n,\Ga o_Y)
    \ge b_{\mathsf P}(x_n)-b_{\mathsf P}(o_Y),
\]
the quasi-isometric lower bound in
Lemma \ref{lem.QIpropsofmetriclike}, applied to the closest orbit
point used to define $h_\psi$, implies
$h_\psi(x_n)\to\infty$.
Set $r_n:=c_2e^{-\d_\psi(o_Y,x_n)}$, where $c_2$ is the uniform
constant in \eqref{eqn.compareshadowsandballs}.
Then $r_n\to0$, and \eqref{eqn.Hmeasurecusp} gives
\[
    \frac{\nu(B_\psi(\xi,r_n))}{r_n}
    \gg
    e^{(2\delta_\psi(\mathsf P)-1) h_\psi(x_n)}
    (C_\psi+h_\psi(x_n))^{a_\psi(\mathsf P)}
    \to\infty.
\]
Hence the Borel set
\[
    E_\infty
    :=\left\{\xi\in\La_\theta:
        \limsup_{r\to0}
        \frac{\nu(B_\psi(\xi,r))}{r}=\infty
      \right\}
\]
has full $\nu$-measure. For any $M>0$ and $\varepsilon>0$, every point of $E_\infty$
is the center of arbitrarily small balls with $r<\varepsilon$
and $\nu(B_\psi(\xi,r))\ge Mr$.
Applying the same Vitali argument as in the finiteness proof yields
\[
    \cal H_{\psi,D\varepsilon}(E_\infty)
    \le D\sum_n r_n
    \le \frac{D}{M}\sum_n\nu(B_\psi(\xi_n,r_n))
    \le \frac{D}{M},
\]
where $D$ is independent of $M$ and $\varepsilon$.
Letting $\varepsilon\to0$ and then $M\to\infty$ gives
$\cal H_\psi(E_\infty)=0$.
If $\cal H_\psi(\La_\theta)$ were positive, it would be finite
by the first part of the proof, and
\eqref{eqn.Hmeasureuniqueness} would imply
\[
    \cal H_\psi(E_\infty)
    =\cal H_\psi(\La_\theta)\nu(E_\infty)
    =\cal H_\psi(\La_\theta)>0,
\]
a contradiction. Therefore $\cal H_\psi(\La_\theta)=0$.

We now consider arbitrary $\psi$. Since $\L_f=\i(\L_f)$,
$\bar\psi$ is positive on $\L_f-\{0\}$. The length bounds in
Lemma~\ref{lem.QIpropsofmetriclike}, together with
$\delta_\psi(\Ga)=1$, imply $\delta_0:=\delta_{\bar\psi}(\Ga)>0$.
Moreover, $\delta_{\psi\circ\i}(\Ga)=\delta_\psi(\Ga)=1$, so
\cite[Corollary 1.9]{CZZ_relative} gives $\delta_0\le1$.
The linear form $\delta_0\bar\psi$ is symmetric, is positive on
$\L_f-\{0\}$, and satisfies $\delta_{\delta_0\bar\psi} (\Ga)=1$.
The symmetric case just proved, applied to $\delta_0\bar\psi$ and its
Patterson--Sullivan measure, gives
$\cal H_{\delta_0\bar\psi}(\La_\theta)<\infty$.
By definition,
\[
    d_{ \delta_0\bar\psi }=d_{\bar\psi}^{\,\delta_0},
    \qquad
    \cal H_{\bar\psi}^{\,\delta_0}=\cal H_{\delta_0\bar\psi}.
\]
If $\delta_0<1$, then for every $\varepsilon>0$,
\[
    \cal H_{\bar\psi,\varepsilon}^{1}(\La_\theta)
    \le \varepsilon^{1-\delta_0}
        \cal H_{\bar\psi,\varepsilon}^{\delta_0}(\La_\theta)
    \le \varepsilon^{1-\delta_0}\cal H_{\delta_0\bar\psi}(\La_\theta).
\]
Letting $\varepsilon\to0$ gives
$\cal H_{\bar\psi}(\La_\theta)=0$.
If $\delta_0=1$, the symmetric case applied to $\bar\psi$ gives
finiteness and precisely the stated peripheral criterion for positivity.
Since $d_\psi\asymp d_{\bar\psi}$ by
Lemma~\ref{lem.gromovproductnearest}, their one-dimensional
Hausdorff measures are comparable. This proves the general assertion.

Finally, if $\Ga$ is Zariski dense and $\psi$ is nonsymmetric,
\cite[Corollary 1.10]{CZZ_relative}, applied to $\psi$ and
$\psi\circ\i$, gives $\delta_0<1$. If $\psi$ is symmetric, then
$\delta_0=\delta_\psi(\Ga)=1$ and $\bar\psi=\psi$.
This proves the last assertions.
\end{proof}

\begin{proof}[Proof of Theorem~\ref{thm.Hmeasure}]
Suppose first that $\psi$ is symmetric.
By Proposition~\ref{prop.Hmeasurepositivefinite},
condition~(2) is equivalent to
$0<\cal H_\psi(\La_\theta)<\infty$.
Together with \eqref{eqn.Hmeasureuniqueness}, this proves
the required equivalence without a Zariski-density assumption.

Now suppose that $\Ga$ is Zariski dense and that
$\psi\ne\psi\circ\i$. We show, more generally, that $\nu$
is not comparable to $\cal H_\psi^s$ for any $s>0$.
Suppose otherwise. By Lemma~\ref{lem.gromovproductnearest},
$d_\psi\asymp d_{\bar\psi}$, and hence
\[
    \cal H_{\bar\psi}^s\asymp\cal H_\psi^s\asymp\nu,
    \qquad
    0<\cal H_{\bar\psi}^s(\La_\theta)<\infty.
\]
Since $d_{s\bar\psi}=d_{\bar\psi}^{\,s}$, the conformality
argument of \cite[Lemma 9.7]{DKO_AR}, applied to the symmetric
form $s\bar\psi$, shows that the normalization of
$\cal H_{\bar\psi}^s$ is $(\Ga,s\bar\psi)$-conformal.
This measure is absolutely continuous with respect to $\nu$.
Since $\nu$ is of divergence type,
\cite[Theorem 1.6]{kim2024conformal} implies
$\psi=s\bar\psi$, contradicting the nonsymmetry of $\psi$.
In particular, $\nu$ cannot be a positive multiple of
$\cal H_\psi$.
\end{proof}

\subsection*{Ahlfors regularity}

\begin{corollary}[Ahlfors regularity]\label{cor.ahlforscriterion}
If $\Ga$ is Zariski dense, the following are equivalent:
\begin{enumerate}
    \item The measure $\nu$ is Ahlfors $s$-regular for some $s>0$;
    \item The measure $\nu$ is Ahlfors $1$-regular;
    \item $\psi=\psi\circ\i$ and, for every $\mathsf P\in\cal P$,
    $\delta_\psi(\mathsf P)=\frac12$ and $a_\psi(\mathsf P)=0$.
\end{enumerate}
If $\psi$ is symmetric, the equivalence also holds without assuming
that $\Ga$ is Zariski dense.
Here Ahlfors $s$-regularity means that there exists $C>1$ such that
\[
    C^{-1}r^s\le\nu(B_\psi(\xi,r))\le Cr^s
    \qquad(\xi\in\La_\theta,\ 0<r\le1).
\]

\end{corollary}

\begin{proof}
Suppose first that $\nu$ is Ahlfors $s$-regular. If $\Ga$ is Zariski dense, Ahlfors $s$-regularity gives
$\nu\asymp\cal H_\psi^s$, so the proof of
Theorem~\ref{thm.Hmeasure} forces $\psi=\psi\circ\i$.
Thus, in either case covered by
the statement, we may use the estimates for symmetric $\psi$.
Choose a conical point $\xi\in \La_\theta$. As in the proof of Lemma \ref{lem.Hmeasurelocalsize}, there are
points $x_i\in[o_Y,\xi]$ in the thick part  of $Y$ with
$d(o_Y,x_i)\to\infty$. Set
\[
    r_i:=e^{-\d_\psi(o_Y,x_i)}.
\]
Then $r_i\to0$, and the ball estimate in
Remark \ref{rmk.ballestimate}, with $h_\psi(x_i)=0$, gives the
two-sided comparison
$\nu(B_\psi(\xi,r_i))\asymp r_i$,
which implies  $s=1$.

Now fix $\mathsf P\in\cal P$ and its parabolic basepoint
$p=\xi_{\mathsf P}$. For $x\in[o_Y,p]$ sufficiently far into its
horoball, the comparison with closest orbit points in
Theorem \ref{thm.generalglobal} and Lemma \ref{lem.closestinhoroball}
give
\[
\begin{aligned}
    h_\psi(x)
    &\approx\d_\psi(\Ga o_Y,x)
     =\d_\psi(x,\Ga o_Y) \approx\d_\psi(x,o_Y)
     =\d_\psi(o_Y,x),
\end{aligned}
\]
where the equalities use $\psi=\psi\circ\i$. Thus $h_\psi(x)=\d_\psi(o_Y,x)+O(1)$ along this ray.
For every sufficiently small $r>0$, choose $x\in[o_Y,p]$ with
$\d_\psi(o_Y,x)=|\log r|+O(1)$ by
Lemma \ref{lem.QIpropsofmetriclike}(3). These points eventually lie in
the horoball, and the preceding comparison gives
$h_\psi(x)=|\log r|+O(1)$. Remark \ref{rmk.ballestimate} therefore
implies
\begin{equation}\label{eqn.paraboliclocalmass}
    \frac{\nu(B_\psi(p,r))}{r}
    \asymp r^{1-2\delta_\psi(\mathsf P)}
            (1+|\log r|)^{a_\psi(\mathsf P)},
\end{equation}
with implied constants independent of $r$.
If $\delta_\psi(\mathsf P)<1/2$, the right-hand side tends to zero.
If $\delta_\psi(\mathsf P)>1/2$, or if
$\delta_\psi(\mathsf P)=1/2$ and $a_\psi(\mathsf P)>0$, it tends to
infinity. Since $a_\psi(\mathsf P)$ is nonnegative, Ahlfors
$1$-regularity forces precisely the condition in (3). This shows $(1) \Rightarrow (2) \Rightarrow (3)$.

Conversely, suppose that (3) holds. Every cusp correction factor in
Remark \ref{rmk.ballestimate} is then equal to $1$. The same estimate
holds in the thick part, so
\[
    \nu(B_\psi(\xi,r))\asymp r
    \qquad(\xi\in\La_\theta,\ 0<r\le1)
\]
with uniform constants. This proves (2), and (2) implies (1).
\end{proof}

\subsection*{Cusped Hitchin representations and the middle simple root}

Recall that a representation
$\rho:\Gamma_0\to\PSL(d,\R)$ of a geometrically finite Fuchsian
group $\Gamma_0$ containing parabolic elements is called
\emph{cusped Hitchin} if there exists a continuous
$\rho$-equivariant positive map
$\xi:\Lambda(\Gamma_0)\to\cal F(\R^d)$,
where $\cal F(\R^d)$ is the full flag variety and positivity
is understood in the sense of Fock--Goncharov; see
\cite[Section 1]{CZZ_cusped}. 
Equivalently, $\rho$ is a type-preserving positive representation;
see \cite[Definition 9.1 and Theorem 9.2]{CZZ_cusped}.

The peripheral condition in Theorem \ref{thm.Hmeasure} can be checked
explicitly for cusped Hitchin representations.

\begin{corollary}\label{cor.cuspedHitchinAhlfors}
Let $\Gamma_0<\PSL(2,\R)$ be a torsion-free nonuniform lattice, let
$m\ge2$, and let
$ \rho:\Gamma_0\to\PSL(2m,\R)$
be a cusped Hitchin representation. Set $\Gamma:=\rho(\Gamma_0)$,
and $\theta=\Pi$.  For
\[
    \psi=\alpha_m,\qquad
    \alpha_m(\operatorname{diag}(t_1,\ldots,t_{2m}))
    :=t_m-t_{m+1},
\]
the Patterson--Sullivan measure $\nu_\psi$ is Ahlfors
$1$-regular with respect to $d_\psi$.
In particular, $\nu_\psi$ is a positive multiple of $\cal H_\psi$.
\end{corollary}

\begin{proof}
Since $\Gamma_0$ is a free group, $\rho$ admits a lift to
$\SL(2m,\R)$, which we use in the Cartan calculations below.
Moreover, $\Gamma$ admits a relatively $\Pi$-Morse embedding $f$;
see \cite[Corollary 1.14]{ZZ_relatively2} and
\cite[Proposition 11.3]{ZZ_relatively}.
 In particular, $\psi$
is positive on $\L_f-\{0\}$. Since $\Gamma_0$ is a lattice, \cite[Corollary 1.7]{CZZ_entropy} gives
$\delta_\psi(\Gamma)=1$.
Moreover, $\psi$ is invariant under $\i$.

Each maximal peripheral subgroup of $\Gamma$ is infinite cyclic.
By \cite[Theorem 1.4(2)]{CZZ_cusped}, its generator $u$ is a regular
unipotent element, up to a central sign in a lift to $\SL(2m,\R)$.
If we put
$H:=\operatorname{diag}(2m-1,2m-3,\ldots,1-2m)$,
then
\[
    \mu(u^n)=(\log|n|)H+O(1)
    \qquad(|n|\to\infty).
\]
Since $\alpha_m(H)=2$, this gives
\[
    \psi(\mu(u^n))=2\log|n|+O(1).
\]
The Poincar\'e series of $\langle u\rangle$ therefore has the same
abscissa of convergence as $\sum_{n\ne0}|n|^{-2s}$. Hence
$\delta_\psi(\langle u\rangle)=\frac12$ and $a_\psi(\langle u\rangle)=0$,
where the second assertion also follows from the virtually cyclic
case of Theorem \ref{thm.countinginsegments}.
Corollary \ref{cor.ahlforscriterion} gives Ahlfors $1$-regularity,
and Theorem \ref{thm.Hmeasure} gives the assertion about $\cal H_\psi$.
\end{proof}

\subsection*{Self-joinings}
Consider a self-joining
$ \Ga=(\rho_1 \times \cdots \times \rho_k)(\Ga_0)
    <\prod_{i=1}^k\mathrm{SO}(n_i,1)^\circ$
as in the introduction: $\Ga_0$ is geometrically finite,
$\rho_1=\mathrm{id}$, and the $\rho_i$ are discrete faithful,
geometrically finite, type-preserving representations. Here $\theta=\Pi$.
We have
$\mu(g)=(t_1,\ldots,t_k)$ with
$t_i=d_{\mathbb H^{n_i}}(o_i,g_i o_i)$. Write
\[
    \psi(t_1,\ldots,t_k)=\sum_{i=1}^k b_i t_i,
    \qquad b_i>0,\qquad \delta_\psi(\Ga)=1,
    \qquad B_\psi:=\sum_{i=1}^k b_i.
\]
In this case, $\i$ is trivial and hence $\psi$ is symmetric.
For a peripheral subgroup $\mathsf P$, let $r(\mathsf P)$ be the
rank of a finite-index free abelian subgroup, and set
$r_{\max}:=\max_{\mathsf P\in\cal P}r(\mathsf P)$, with
$r_{\max}=0$ if $\cal P$ is empty.

\begin{corollary}\label{cor.realhyperbolicselfjoinings}
For every $\mathsf P\in\cal P$,
\[
    \delta_\psi(\mathsf P)=\frac{r(\mathsf P)}{2B_\psi},
    \qquad a_\psi(\mathsf P)=0.
\]
Consequently, $\nu$ is a positive multiple of $\cal H_\psi$ if and
only if $B_\psi\ge r_{\max}$.
Moreover, $\nu$ is Ahlfors $1$-regular with respect to $d_\psi$ if and only if
$r(\mathsf P)=B_\psi$ for every $\mathsf P\in\cal P$.
\end{corollary}
\begin{proof}
Write $r=r(\mathsf P)$. Passing to a finite-index subgroup
does not change the counting exponents. We may therefore assume
that $\mathsf P=\{\gamma_n:n\in\mathbb Z^r\}\simeq \mathbb Z^r$
and that, in each factor, it acts by translations on an invariant
$r$-dimensional affine subspace of a horosphere.

Choose a point $q_i$ in this subspace. 
Normalizing the horosphere to have height $1$, the hyperbolic
distance formula gives
\[
    d_{\mathbb H^{n_i}}(q_i,(\gamma_n)_i q_i)
    =2\log(1+\|n\|)+O(1).
\]
Replacing $q_i$ by $o_i$ changes this distance by at most
$2d_{\mathbb H^{n_i}}(o_i,q_i)$. Hence
\[
    \psi(\mu(\gamma_n))
    =2B_\psi\log(1+\|n\|)+O(1),
\]
and counting lattice points gives
\[
    \#\{\gamma\in\mathsf P:\psi(\mu(\gamma))\le T\}
    \asymp e^{rT/(2B_\psi)}.
\]
This proves the formulas for $\delta_\psi(\mathsf P)$ and
$a_\psi(\mathsf P)$. The remaining assertions follow from
Theorem~\ref{thm.Hmeasure} and
Corollary~\ref{cor.ahlforscriterion}.
\end{proof}
\begin{example}[Cusped surface self-joinings]
\label{ex.cuspedsurfaceselfjoining}
Let $S$ be a once-punctured torus, and let
$j_1,j_2:\pi_1(S)\to\mathrm{PSL}_2(\R)$ be two
discrete faithful finite-area holonomies that are not conjugate
in $\mathrm{PGL}_2(\R)$. Set
$\Ga=(j_1 \times j_2)(\pi_1(S))$. For $0<s<1$, define
\[
    q_s(t_1,t_2):=s\cdot t_1+(1-s)\cdot t_2,
    \qquad \psi:=\delta_{q_s}(\Ga)\,q_s.
\]
The coordinate critical exponents are both $1$.
Since $j_1$ and $j_2$ have different marked length spectra,
\cite[Corollary 1.9]{CZZ_relative} gives
$\delta_{q_s}(\Ga)<1$.
By construction, $\delta_\psi(\Ga)=1$, every peripheral
subgroup has rank $1$, and $B_\psi=\delta_{q_s}(\Ga)<1$.
Corollary~\ref{cor.realhyperbolicselfjoinings} and Proposition \ref{prop.Hmeasurepositivefinite} therefore give
$\cal H_\psi(\La_\theta)=0$, although $\nu$ is locally doubling
and locally reverse doubling by Theorem~\ref{thm.doubling}.
\end{example}

%\bibliographystyle{plain} 
%\bibliography{GS}

\end{document}